\documentclass[11pt]{article}

\pdfoutput=1

\title{Finding large counterexamples by selectively exploring the\\
Pachner graph}
\date{\today}

\author{
 Benjamin A. Burton\\
 The University of Queensland\\
 \email{bab}{maths.uq.edu.au}
 \and
 Alexander He\\
 The University of Queensland\\
 \email{a.he}{uqconnect.edu.au}
}

\usepackage{amsmath, amssymb, amsthm, caption, subcaption, xcolor}

\usepackage[classfont=roman,langfont=caps,basic]{complexity}
\renewclass{\coNP}{co\text{-}NP}

\usepackage[shortlabels]{enumitem}
\setlist[itemize]{leftmargin=*, noitemsep}
\setlist[enumerate]{leftmargin=*, noitemsep}
\setlist[description]{leftmargin=*, labelwidth=*, noitemsep}

\usepackage[bottom]{footmisc}

\usepackage[a4paper, margin=2cm]{geometry}

\usepackage{ragged2e}
\usepackage{graphicx}
\graphicspath{ {./TriangCounterexFigures/} }

\usepackage[super]{nth}
\usepackage{siunitx}

\usepackage{tikz}
\usetikzlibrary{cd} 
\usetikzlibrary{calc} 
\usetikzlibrary{arrows.meta} 

\usepackage[backref=page]{hyperref}

\theoremstyle{plain}
\newtheorem{theorem}{Theorem}

\newtheorem{proposition}[theorem]{Proposition}
\newtheorem{construction}[theorem]{Construction}
\newtheorem{observation}[theorem]{Observation}

\newtheorem{algorithm}[theorem]{Algorithm}

\newtheorem{conjecture}[theorem]{Conjecture}
\newtheorem*{theorem*}{Theorem}

\theoremstyle{definition}

\newtheorem{definition}[theorem]{Definition}
\newtheorem{definitions}[theorem]{Definitions}

\newcommand{\Regina}{\texttt{\textup{Regina}}}
\newcommand{\False}{\texttt{\textup{false}}}

\newcommand{\Unknown}{\texttt{\textup{inconclusive}}}
\DeclareMathOperator{\image}{\mathrm{im}}

\newcommand\simtimes{\mathbin{%
    \stackrel{\sim}{\smash{\times}\rule{0pt}{0.6ex}}%
}}

\newcommand{\ConnSum}{\mathbin{\#}}

\makeatletter
\renewenvironment{proof}[1][\proofname] {\par\pushQED{\qed}\normalfont\topsep6\p@\@plus6\p@\relax\trivlist\item[\hskip\labelsep\itshape\bfseries#1\@addpunct{.}]\ignorespaces}{\popQED\endtrivlist\@endpefalse}
\makeatother

\newcommand{\email}[2]{\href{mailto:#1@#2}{\textsf{#1\hspace{1pt}$@$\hspace{1pt}#2}}}

\tikzset{
	resultsbox/.style={
		every matrix/.append style={
			draw=black,
			fill=cyan!20,
			ultra thick,
			rounded corners,
			inner sep=3pt
		}
	}
}

\begin{document}

\maketitle

\begin{abstract}
We often rely on censuses of triangulations to guide our intuition in $3$-manifold topology.
However, this can lead to misplaced faith in conjectures if
the smallest counterexamples are too large to appear in our census.
Since the number of triangulations increases super-exponentially with size,
there is no way to expand a census beyond relatively small triangulations---the
current census only goes up to $10$ tetrahedra.
Here, we show that it is feasible to search for large and hard-to-find counterexamples by
using heuristics to \emph{selectively} (rather than exhaustively) enumerate triangulations.
We use this idea to find counterexamples to three conjectures
which ask, for certain $3$-manifolds, whether one-vertex triangulations always
have a ``distinctive'' edge that would allow us to recognise the $3$-manifold.
\end{abstract}
\paragraph{Acknowledgements}We thank the referees for their helpful comments.
The second author was supported by an Australian Government Research Training Program Scholarship.

\section{Introduction}\label{sec:intro}

There is no shortage of conjectures in computational geometry and topology that are true in small cases,
but which turn out to be false in general due to the existence of a relatively large counterexample.
Perhaps the most notable example of this is the \emph{Hirsch conjecture}, which posited that
in a $d$-dimensional polytope with $n$ facets, any two vertices can be connected by a path of at most $n-d$ edges.
It is known that the Hirsch conjecture is true when the dimension is small (specifically, when $d\leqslant3$~\cite{Klee1966}),
as well as when the number of facets is small (specifically, when $n\leqslant d+6$~\cite{BremnerSchewe2011}).
However, these small cases are not indicative of the general behaviour:
Santos showed that there is a $43$-dimensional counterexample with $86$ facets~\cite{Santos2012}.

For an example from computational $3$-manifold topology, consider the Seifert fibre spaces
(see Section~\ref{subsec:sfs} for a definition of these $3$-manifolds).
There are $302$ Seifert fibre spaces that can be triangulated with no more than $7$ tetrahedra;
for all of these, at least one \textbf{minimal triangulation}
(triangulation with the smallest possible number of tetrahedra)
uses a standard prism-and-layering construction.
This pattern does not persist: there is a Seifert fibre space whose unique ($8$-tetrahedron) minimal triangulation is
given by the isomorphism signature
\texttt{iLLLPQcbcgffghhhtsmhgosof}
(the software package \Regina~\cite{Regina} can convert this string back into a triangulation),
and this triangulation is instead constructed from a gadget---called the \textbf{brick} $B_5$---that was first identified
in Martelli and Petronio's census of minimal triangulations up to $9$ tetrahedra~\cite{MartelliPetronio2001}.\footnote{
Although the first author~\cite{Burton2011ISSAC} and Matveev~\cite{MatveevAtlas} have since
extended the census, no other gadgets like $B_5$ have yet been found.}

Our main goal in this paper is to present a technique for finding large counterexamples
in a similar setting to the one just mentioned.
The difference is that instead of only studying properties of minimal triangulations,
we consider arbitrary one-vertex triangulations.
The practical motivation for this is that one-vertex triangulations are easy to obtain,
whereas it is generally difficult to know for sure whether a particular triangulation is minimal.

In our setting, the most obvious source of counterexamples would be a census
of all possible triangulations up to a given number of tetrahedra.
However, this only captures small counterexamples because the number of triangulations
increases super-exponentially as we increase the number of tetrahedra;
to date, the census of all possible closed $3$-manifold triangulations only goes up to $10$ tetrahedra,
and this already includes over $2$~billion triangulations, constituting over $63$~GB of data~\cite{Burton2011ISSAC}.

How can we find a counterexample that is too large to appear in the census?
We showcase a method of \emph{selectively} (rather than exhaustively) enumerating triangulations,
which allows us to find large counterexamples to three related conjectures
posed by Saul Schleimer at the 2022 Dagstuhl workshop on
\emph{Computation and Reconfiguration in Low-Dimensional Topological Spaces}.

\subsection{The conjectures}\label{subsec:conj}

Each of the three conjectures mentioned above concerns properties of edges of one-vertex triangulations.
Since our triangulations only have one vertex, each edge realises an embedded closed curve,
so we can ask whether such an edge is embedded in an interesting or useful way.

For instance, recall that a \textbf{lens space} is a $3$-manifold given by gluing together two solid tori.
One way to recognise a lens space is to find a \textbf{core curve}:
an embedded closed curve $\gamma$ such that the complement of an open regular neighbourhood of $\gamma$ is a solid torus.
We can ask whether one-vertex triangulations of lens spaces
always contain a \textbf{core edge} (i.e., an edge that realises a core curve):

\begin{conjecture}\label{conj:core}
Every one-vertex triangulation of a lens space has a core edge.
\end{conjecture}

Since solid torus recognition can be solved efficiently in practice~\cite{BurtonOzlen12,Burton2013Regina},
proving Conjecture~\ref{conj:core} would have provided a new and relatively efficient method for recognising lens spaces.
Lens space recognition can also be used to determine whether a 3-manifold is \textbf{elliptic}
(i.e., admits spherical geometry).
Indeed, Lackenby and Schleimer~\cite{LackenbySchleimer2022} recently showed
that recognising elliptic manifolds is in $\NP$;
their proof relies on the following result, which can be viewed as a variant of Conjecture~\ref{conj:core}:

\begin{theorem*}[Lackenby and Schleimer~{\cite[Theorem~9.4]{LackenbySchleimer2022}}]
Let $\mathcal{M}$ be a lens space that is neither $\mathbb{R}P^3$ nor a prism manifold\footnote{
Lackenby and Schleimer use a slightly different definition of prism manifold than the one we give in Section~\ref{subsec:sfs}.},
and let $\mathcal{T}$ be any triangulation of $\mathcal{M}$.
Then the \nth{86} iterated barycentric subdivision of $\mathcal{T}$
contains a sequence of edges that forms a core curve.
\end{theorem*}

The other two conjectures have similar motivations:
if every one-vertex triangulation of a particular type of 3-manifold contains
an edge with a distinctive property, then we can use such an edge to recognise the $3$-manifold.

For the second conjecture, consider a knot $K$ in the 3-sphere.
The \textbf{tunnel number} of $K$ is the smallest number of arcs that need to be added to $K$ so that
the complement becomes a handlebody; by construction, this is a knot invariant.
The three knots shown in Figure~\ref{fig:tunnelOneKnots} are all examples of knots with tunnel number equal to one
(though this might not be obvious at first glance).

\begin{figure}[htbp]
\centering
	\begin{subfigure}[t]{0.3\textwidth}
	\centering
	\includegraphics[scale=0.8]{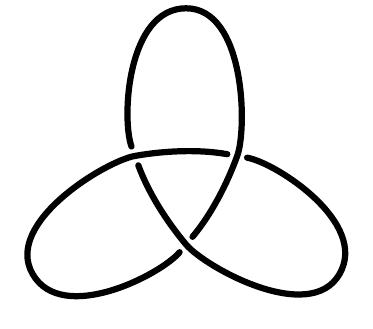}
	\caption{The trefoil knot.}
	\end{subfigure}
	\hfill
	\begin{subfigure}[t]{0.3\textwidth}
	\centering
	\includegraphics[scale=0.8]{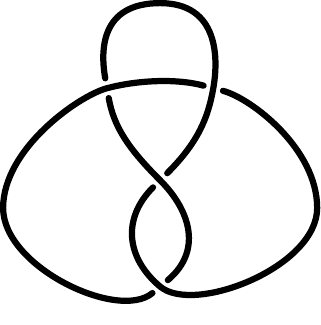}
	\caption{The figure-eight knot.}
	\end{subfigure}
	\hfill
	\begin{subfigure}[t]{0.3\textwidth}
	\centering
	\includegraphics[scale=0.8]{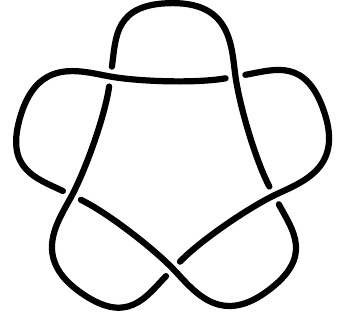}
	\caption{The $(5,2)$ torus knot.}
	\end{subfigure}
\caption{Three knots with tunnel number equal to one.}
\label{fig:tunnelOneKnots}
\end{figure}

Let $\mathcal{T}$ be a one-vertex ideal triangulation of the complement of $K$,
so that each edge $e$ in $\mathcal{T}$ forms an arc $\alpha$ that meets $K$ at its endpoints (and nowhere else).
We call $e$ a \textbf{tunnel edge} if the complement of $K\cup\alpha$ is a genus-2 handlebody;
assuming that $K$ is not the unknot, the existence of a tunnel edge implies that $K$ has tunnel number equal to one.

\begin{conjecture}\label{conj:tunnel}
Let $K$ be a knot with tunnel number equal to one.
Then every one-vertex ideal triangulation of the complement of $K$ has a tunnel edge.
\end{conjecture}

For the third conjecture, we consider \textbf{Seifert fibre spaces}.
These are 3-manifolds that are fibred by circles in a particular way
(see Section~\ref{subsec:sfs} for a complete definition);
we call each circle fibre a \textbf{Seifert fibre}, and we call an edge a \textbf{Seifert fibre edge} if it is isotopic to a Seifert fibre.
Seifert fibre spaces are important because they play a central role in
canonical torus decompositions for closed orientable irreducible $3$-manifolds:
\begin{itemize}
\item The JSJ decomposition
(due to Jaco and Shalen~\cite{JacoShalen1978,JacoShalen1979}, and Johannson~\cite{Johannson1979})
cuts such $3$-manifolds into pieces that are either atoroidal or Seifert fibred.
\item The geometric decomposition from Thurston's geometrisation conjecture
(famously proved by
Perelman~\cite{Perelman2002,Perelman2003March,Perelman2003July})
cuts such $3$-manifolds into pieces that are either hyperbolic or graph manifolds;
the graph manifolds are precisely the $3$-manifolds that can be further decomposed along tori into Seifert fibre spaces.
\end{itemize}
The \textbf{small} Seifert fibre spaces do not contain any embedded two-sided incompressible surfaces,
which makes them relatively difficult to work with.

\begin{conjecture}\label{conj:fibre}
Every one-vertex triangulation of a small Seifert fibre space has a Seifert fibre edge.
\end{conjecture}

\subsection{Using a targeted search to find counterexamples}\label{subsec:counterex}

As mentioned earlier, we have found counterexamples to all three conjectures listed in Section~\ref{subsec:conj}.
Our focus for the rest of this paper will mainly be on the methods that we used to find these counterexamples.
The key innovation is a heuristic for measuring how ``far away'' a triangulation is from having no core edges
(or more generally, no edges satisfying any particular topological property that is computable).

To see why such a heuristic is necessary, consider the $3$-sphere---the simplest possible lens space;
for the $3$-sphere, core edges are equivalent to edges that are \textbf{unknotted}.
In the census of all triangulations with up to $10$ tetrahedra,
there are \num{422533279} one-vertex triangulations of the 3-sphere.
We tested all of these and found that they all have at least one core edge.
This required just over $22$ hours of wall time running on $12$ threads in parallel,
but this does not include the time required to:
\begin{enumerate}[label={(\arabic*)}]
\item generate the census in the first place; and
\item identify all the $3$-spheres in this census (which had been done
previously~\cite{Burton2011ISSAC,Burton2011arXiv}).
\end{enumerate}
(We also performed a more exhaustive test---requiring about $53$ and a half hours
of wall time on $12$ parallel threads---which showed that all of the one-vertex $3$-spheres
with up to $10$ tetrahedra actually have at least \emph{two} core edges.)

The main takeaway is that this exhaustive search was both expensive and unsuccessful.
In contrast, by using our heuristic to enumerate triangulations in a targeted fashion,
we were able to find all of our counterexamples in just \emph{minutes} of wall time,
without ever needing to pre-process large quantities of data.
It is also worth noting that all of our counterexamples have at least $11$ tetrahedra,
putting them beyond the limits of the census.
See Section~\ref{sec:results} for detailed computational results;
at the end of Section~\ref{sec:results}, we also show how to turn each of our counterexamples into an infinite family.

We discuss the key tools and algorithms that allowed us to find these examples in
Sections~\ref{sec:tools} and~\ref{sec:heuristics}.
Specifically, Section~\ref{sec:heuristics} introduces our heuristic and our targeted search algorithm.
Section~\ref{sec:tools} discusses a number of auxiliary algorithms, all of which may be of independent interest;
in particular, we present (and implement) an improved algorithm for
handlebody recognition in Section~\ref{subsec:handlebody}.

We finish with Section~\ref{sec:discussion}, which raises some questions that remain unanswered.
In particular, in Section~\ref{subsec:future}, we mention some other potential applications
of the idea of using heuristics to selectively enumerate large triangulations.

\section{Preliminaries}\label{sec:prelims}

Throughout this paper, we will say that a compact $3$-manifold is \textbf{closed} if it has empty boundary, and \textbf{bounded} if it has non-empty boundary.

\subsection{Triangulations}\label{subsec:triang}

A \textbf{(generalised) triangulation} $\mathcal{T}$
is a finite collection of tetrahedra whose triangular faces may be affinely identified in pairs;
each equivalence class of such faces is a \textbf{face} of $\mathcal{T}$.
The \textbf{size} of $\mathcal{T}$, denoted $\lvert\mathcal{T}\rvert$,
is the number of tetrahedra that make up $\mathcal{T}$.
The \textbf{boundary} of $\mathcal{T}$ consists of all the faces that
are not identified with any other faces.
We allow faces of the same tetrahedron to be identified,
which means that our triangulations need not be simplicial complexes.

As a result of the face identifications, many edges and vertices of the tetrahedra may also become identified.
The resulting equivalence classes of edges are called \textbf{edges} of the triangulation;
similarly, the resulting equivalence classes of vertices are called \textbf{vertices} of the triangulation.
The \textbf{degree} of an edge $e$ of a triangulation is the number of tetrahedra that meet $e$, counted with multiplicity;
this is equivalent to the number of tetrahedron edges that are identified to form $e$.
The \textbf{1-skeleton} of a triangulation $\mathcal{T}$ is the subcomplex
consisting only of the vertices and edges of $\mathcal{T}$.

If a point lies on the boundary of a triangulation, then we say that it is \textbf{boundary};
otherwise, we say that it is \textbf{internal}.
If an edge or face consists \emph{entirely} of boundary points, then we say that it is \textbf{boundary};
otherwise, we say that it is \textbf{internal}.

A notable feature of generalised triangulations is that it is possible to get a \textbf{one-vertex triangulation}:
a triangulation in which all of the tetrahedron vertices are identified to form a single vertex.
In fact, \emph{every} compact $3$-manifold with at most one boundary component
admits a one-vertex triangulation;
one way to prove this is to use the two-tetrahedron ``triangular-pillow-with-tunnel''
that was first introduced by Weeks~\cite{SnapPy,Weeks2005}.

As a small but informative example, we build a one-vertex triangulation $\mathcal{T}$
of the solid torus using just a single tetrahedron $\Delta$.
Label the vertices of $\Delta$ by $0$, $1$, $2$ and $3$, as shown in Figure~\ref{fig:oneTetLST}.
Glue the front two faces so that vertices $0$, $1$ and $2$ of the left face are
respectively identified with vertices $1$, $3$ and $0$ of the right face.
Although $\mathcal{T}$ is not a simplicial complex, it is a perfectly sensible
triangulation in the sense we have just defined.
Observe that all four vertices of $\Delta$ get identified to become a single boundary vertex in $\mathcal{T}$,
and that $\mathcal{T}$ has three edges (all of which are boundary) given by identifying the following edges of $\Delta$:
\begin{itemize}
\item edges $20$, $01$ and $13$ all form a single edge, as indicated by the black arrows in Figure~\ref{fig:oneTetLST};
\item edges $12$ and $30$ together form a single edge, as indicated by the white arrows in Figure~\ref{fig:oneTetLST}; and
\item edge $23$ forms the final edge.
\end{itemize}
One way to see that $\mathcal{T}$ does indeed triangulate the solid torus is
to observe that the internal face of $\mathcal{T}$ forms a M\"{o}bius band;
we can then think of the whole triangulation as a ``thickening'' of this M\"{o}bius band, which is a solid torus.

\begin{figure}[htbp]
\centering
\begin{tikzpicture}[scale=1]
\node[inner sep=0pt] (LST) at (0,0)
	{\includegraphics[scale=0.75]{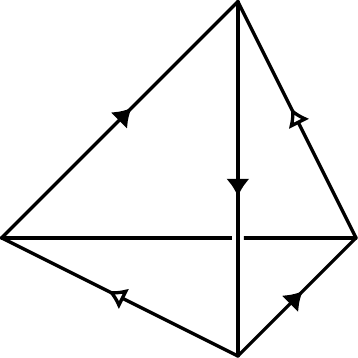}};

\node[above, inner sep=3pt] at ($(LST.north)+(0.75,0)$) {$0$};
\node[below, inner sep=3pt] at ($(LST.south)+(0.75,0)$) {$1$};
\node[left, inner sep=3pt] at ($(LST.west)+(0,-0.75)$) {$2$};
\node[right, inner sep=3pt] at ($(LST.east)+(0,-0.75)$) {$3$};
\end{tikzpicture}
\caption{Gluing two faces of a single tetrahedron with a ``twist'' gives a one-vertex triangulation of the solid torus.}
\label{fig:oneTetLST}
\end{figure}
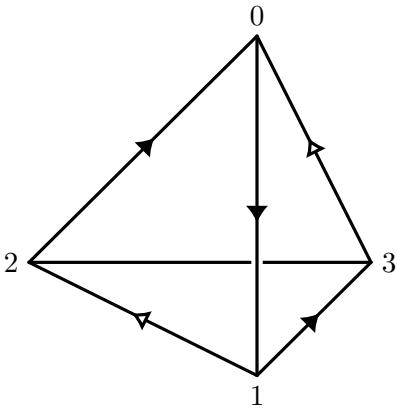

In general, for a triangulation $\mathcal{T}$ to represent a $3$-manifold, every point $p$ in $\mathcal{T}$ must
be \textbf{non-singular} in the sense that it has a small neighbourhood bounded by either
a disc (if $p$ is boundary) or
a $2$-sphere (if $p$ is internal).
This condition may fail to hold for vertices and midpoints of edges.
For this paper, we insist that no edge is ever identified with itself in reverse;
this ensures that midpoints of edges are always non-singular.

We call $\mathcal{T}$ a \textbf{3-manifold triangulation} if all its vertices are non-singular,
because in this case the underlying topological space of $\mathcal{T}$ will genuinely be a $3$-manifold (possibly with boundary).
We call a $3$-manifold triangulation $\mathcal{T}$ \textbf{closed} if the underlying $3$-manifold $\mathcal{M}$
is closed, and we call $\mathcal{T}$ \textbf{bounded} if $\mathcal{M}$ is bounded.

In addition to $3$-manifold triangulations, we also work with triangulations containing vertices that are \textbf{ideal},
meaning that a small neighbourhood of the vertex is bounded by a closed surface other than the $2$-sphere.
For this paper, we will call a triangulation \textbf{ideal} if \emph{all} of its vertices are ideal.
Ideal triangulations often give an efficient way to represent the bounded $3$-manifold given by truncating the ideal vertices;
this idea originated with Thurston's two-tetrahedron ideal triangulation of
the figure-eight knot complement~\cite[Example~1.4.8]{Thurston1997}.

To concretely encode a triangulation,
we usually give each tetrahedron both a label and an ordering of its four vertices.
Two triangulations are \textbf{(combinatorially) isomorphic} if they are identical up to
relabelling tetrahedra and/or reordering tetrahedron vertices.
We can uniquely identify any isomorphism class of triangulations using
a short, efficiently-computable string called an \textbf{isomorphism signature}.
There are many ways to formulate isomorphism signatures;
we use the formulation described in \cite{Burton2011arXiv},
which is implemented in \Regina~\cite{Burton2013Regina,Regina}.

\subsection{The 2-3 and 3-2 moves}\label{subsec:elemMoves}

Given any one-vertex triangulation $\mathcal{T}$ with at least two tetrahedra,
we can produce a new one-vertex triangulation of the same $3$-manifold by
performing a \textbf{2-3 move} about a (triangular) face that meets two distinct tetrahedra;
this move replaces these two tetrahedra with three tetrahedra attached around a new edge $e$,
as illustrated in Figure~\ref{fig:pachnerMove}.
We call the inverse move a \textbf{3-2 move} about the edge $e$.
Observe that it is possible to perform a 3-2 move about an edge if and only if
this edge is an internal degree-$3$ edge that actually meets three distinct tetrahedra.

\begin{figure}[htbp]
\centering
	\begin{tikzpicture}
	\node[inner sep=0pt] (Before) at (0,0)
		{\includegraphics[scale=0.6]{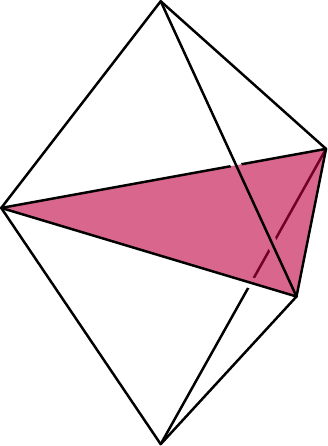}};
	\node[inner sep=0pt] (After) at (5.5,0)
		{\includegraphics[scale=0.6]{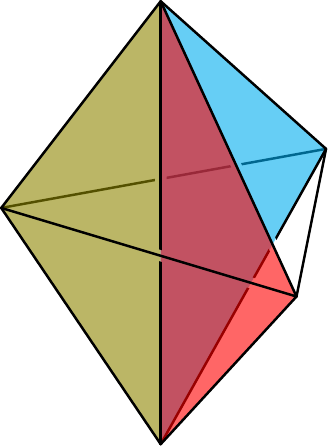}};

	\begin{scope}[very thick, line cap=round, -{Stealth}]
	\draw ($(Before.east)+(0,0.3)$)
		-- ($(After.west)+(0,0.3)$)
		node[midway, above, inner sep=3pt] {2-3};
	\draw ($(After.west)+(0,-0.3)$)
		-- ($(Before.east)+(0,-0.3)$)
		node[midway, below, inner sep=3pt] {3-2};
	\end{scope}
	\end{tikzpicture}
\caption{The 2-3 and 3-2 moves.}
\label{fig:pachnerMove}
\end{figure}

With this in mind, consider a closed $3$-manifold $\mathcal{M}$.
We can think of the one-vertex triangulations of $\mathcal{M}$ with at least two tetrahedra as
nodes of an infinite graph, with two nodes connected by an undirected arc if and only if
the corresponding triangulations are related by a 2-3 move;
this graph is called the \textbf{Pachner graph} of $\mathcal{M}$.
Matveev~\cite{Matveev2007} and Piergallini~\cite{Piergallini1988}
independently proved the following result (though they did not state this result in terms of Pachner graphs):

\begin{theorem}\label{thm:pachnerConnected}
For every closed 3-manifold $\mathcal{M}$, the Pachner graph of $\mathcal{M}$ is connected.
\end{theorem}

An analogous result holds for one-vertex ideal triangulations,
as a special case of a theorem due to Amendola~\cite{Amendola2005}.
An alternative proof of Amendola's theorem can also be found in~\cite{RST2019}.

Since we are particularly interested in the \emph{edges} of one-vertex triangulations,
it is worth emphasising how the 2-3 and 3-2 moves affect the $1$-skeleton of a triangulation.
A 2-3 move creates a new edge, but otherwise leaves the $1$-skeleton untouched.
A 3-2 move about an edge $e$ removes this edge, but otherwise leaves the $1$-skeleton untouched.

\subsection{Embedded surfaces}\label{subsec:embSurfs}

We can often learn a lot about a $3$-manifold by looking at the embedded surfaces that it contains.
Here, we review some standard terminology for such surfaces.

Throughout this section, let $S$ denote a surface embedded in
a compact $3$-manifold $\mathcal{M}$ (possibly with boundary).
We say that $S$ is \textbf{properly} embedded in $\mathcal{M}$ if $S\cap\partial\mathcal{M} = \partial S$.

If $S$ is a $2$-sphere, then we say that $S$ is \textbf{inessential} if it
bounds a $3$-ball in $\mathcal{M}$, and \textbf{essential} otherwise.
The $3$-manifold $\mathcal{M}$ is \textbf{reducible} if it
contains an essential $2$-sphere, and \textbf{irreducible} otherwise.

A disc $D$ embedded in $\mathcal{M}$ is a \textbf{compression disc} for $S$ if $D\cap S=\partial D$.
Such a compression disc $D$ is \textbf{inessential} if the curve $\partial D$ bounds
another disc that lies entirely inside $S$; otherwise, $D$ is \textbf{essential}.
Figure~\ref{fig:InessentialCompressingDisc} illustrates an inessential compression disc.
The surface $S$ is \textbf{compressible} if it admits an essential compression disc, and \textbf{incompressible} otherwise.
In the special case where the surface $S$ is just $\partial\mathcal{M}$,
notice that a compression disc for $S$ is the same thing as a properly embedded disc in $\mathcal{M}$;
thus, whenever we are working with a \emph{properly} embedded disc $D$,
we will also often refer to $D$ as a compression disc (without needing to specify the surface $\partial\mathcal{M}$).
If $\partial\mathcal{M}$ is compressible, then we say that
$\mathcal{M}$ is \textbf{$\partial$-reducible}; otherwise, $\mathcal{M}$ is \textbf{$\partial$-irreducible}.

\begin{figure}[htbp]
\centering
\begin{tikzpicture}
\node[inner sep=0pt] at (0,0) {
	\includegraphics[scale=1.25]{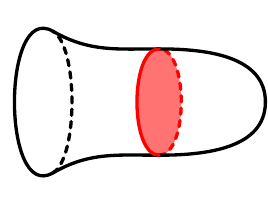}};

\node at (-0.25,0) {\Large\textcolor{red}{$D$}};
\node at (-2.7,0.85) {\Large $S$};
\end{tikzpicture}
\caption{An inessential compression disc $D$ for a surface $S$.}
\label{fig:InessentialCompressingDisc}
\end{figure}

We say that $\mathcal{M}$ is \textbf{sufficiently large} if it contains a
two-sided incompressible surface other than $S^2$ (the $2$-sphere) or $\mathbb{R}P^2$ (the real projective plane);
otherwise, $\mathcal{M}$ is \textbf{small}.
We say that $\mathcal{M}$ is \textbf{Haken} if it is irreducible, $\partial$-irreducible and sufficiently large;
otherwise, $\mathcal{M}$ is \textbf{non-Haken}.
The terms ``small'' and ``non-Haken'' are sometimes used interchangeably because it is very common
to work exclusively with $3$-manifolds that are irreducible and $\partial$-irreducible.

Suppose, for this paragraph, that $S$ has non-empty boundary and is properly embedded.
A disc $D$ embedded in $\mathcal{M}$ is a \textbf{$\partial$-compression disc} for $S$ if there are two arcs
$\alpha$ and $\beta$ such that $\alpha=D\cap S$, $\beta=D\cap\partial\mathcal{M}$,
$\alpha\cup\beta=\partial D$, and $\alpha\cap\beta=\partial\alpha=\partial\beta$.
Such a $\partial$-compression disc $D$ is \textbf{inessential} if there is another arc $\gamma$ in $\partial S$ such that
$\alpha\cap\gamma=\partial\alpha=\partial\gamma$ and the curve $\alpha\cup\gamma$ bounds another disc that lies entirely inside $S$;
otherwise, $D$ is \textbf{essential}.
Figure~\ref{fig:InessentialBdryCompressingDisc} illustrates an inessential $\partial$-compression disc.
The surface $S$ is \textbf{$\partial$-compressible} if it admits an essential $\partial$-compression disc,
and \textbf{$\partial$-incompressible} otherwise.

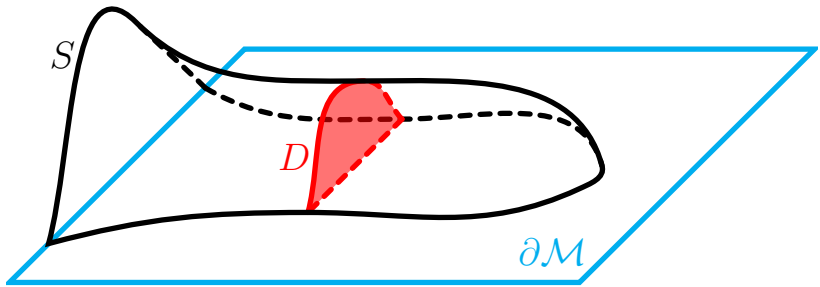
\begin{figure}[htbp]
\centering
\begin{tikzpicture}
\node[inner sep=0pt] at (0,0) {
	\includegraphics[scale=1.25]{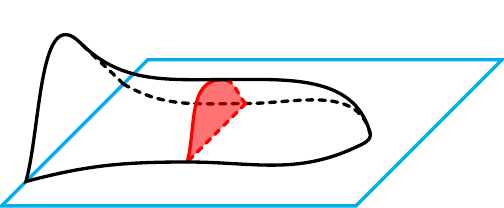}};

\node at (-1.55,-0.5) {\Large\textcolor{red}{$D$}};
\node at (1.85,-1.75) {\Large\textcolor{cyan}{$\partial\mathcal{M}$}};
\node at (-4.6,0.85) {\Large $S$};
\end{tikzpicture}
\caption{An inessential $\partial$-compression disc $D$ for a surface $S$.}
\label{fig:InessentialBdryCompressingDisc}
\end{figure}

Finally, suppose $S$ is properly embedded.
We say that $S$ is \textbf{boundary parallel} if there is an isotopy fixing $\partial S$ that sends $S$ to a subsurface of $\partial\mathcal{M}$.
In the case where $S$ is neither a $2$-sphere nor a disc, we say that $S$ is \textbf{essential} if
it is incompressible, $\partial$-incompressible, and not boundary parallel.

\subsection{Normal surfaces}\label{subsec:normSurfs}

As we discuss in Sections~\ref{subsec:handlebody} and~\ref{subsec:edgeSeifert},
our work relies on a number of results from normal surface theory.
Here, we only outline the ideas that are required for our purposes;
see~\cite{HLP1999} and~\cite{JacoRubinstein2003} for more comprehensive overviews.

A (possibly disconnected) properly embedded surface $S$ in a $3$-manifold triangulation $\mathcal{T}$ is a \textbf{normal surface} if:
\begin{itemize}
\item $S$ meets each simplex (i.e., vertex, edge, triangle, or tetrahedron) of $\mathcal{T}$ transversely; and
\item $S$ meets each tetrahedron $\Delta$ of $\mathcal{T}$ in a finite
(and possibly empty) collection of discs---known
as \textbf{elementary discs}\footnote{Some authors use the term ``normal disc'' (for instance, see~\cite{JacoRubinstein2003}).
We reserve the term ``normal disc'' to refer to a normal surface that forms a properly embedded disc.}---where
each such disc is a curvilinear triangle or quadrilateral whose vertices lie on different edges of $\Delta$.
\end{itemize}
Up to \textbf{normal isotopy} (i.e., an ambient isotopy that preserves every simplex of $\mathcal{T}$),
every elementary disc has one of the seven types shown in Figure~\ref{fig:elemDiscs}.

\begin{figure}[htbp]
	\centering
	\begin{subfigure}[t]{0.25\textwidth}
	\centering
	\includegraphics[scale=0.75]{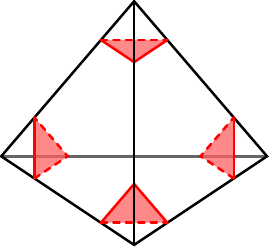}
	\caption{The four triangle types.}
	\end{subfigure}
	\hfill
	\begin{subfigure}[t]{0.7\textwidth}
	\centering
	\begin{tikzpicture}
		\node at (0,0)
			{\includegraphics[scale=0.75]{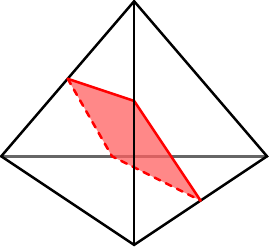}};
		\node at (3.75,0)
			{\includegraphics[scale=0.75]{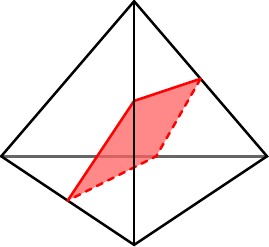}};
		\node at (7.5,0)
			{\includegraphics[scale=0.75]{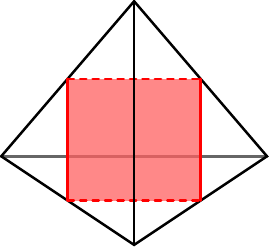}};
	\end{tikzpicture}
	\caption{The three quadrilateral types.}
	\end{subfigure}
	\caption{The seven elementary disc types in a tetrahedron.}
	\label{fig:elemDiscs}
\end{figure}

The simplest example of a normal surface is a \textbf{vertex-linking surface},
which consists entirely of triangles.
Such surfaces always exist, so finding these surfaces never
gives us any new information about the underlying $3$-manifold.
For this reason, we consider a normal surface to be \textbf{non-trivial}
when it includes at least one quadrilateral piece.

If $\mathcal{T}$ is an $n$-tetrahedron triangulation, then we can represent any normal surface in $\mathcal{T}$ as
a \textbf{normal coordinate} vector in $\mathbb{R}^{7n}$ that counts
the number of elementary discs of each type in each tetrahedron.
Conversely, it turns out that the space of vectors in $\mathbb{R}^{7n}$ that
correspond to normal surfaces in $\mathcal{T}$ can be described by a collection of
combinatorial constraints (namely, the \textbf{quadrilateral constraints})
together with a collection of homogeneous linear inequalities.
By homogeneity, every solution to these linear inequalities is a scalar multiple of
a vector in a convex polytope $\mathcal{P}(\mathcal{T})\subset\mathbb{R}^{7n}$.

Given two normal surfaces $S$ and $S'$, we can consider the sum of the two corresponding normal coordinate vectors.
Provided the quadrilateral constraints are still satisfied,
this gives a new normal surface called the \textbf{(Haken) sum} of $S$ and $S'$, denoted $S+S'$.
In particular, we can always take the sum of a normal surface $S$ with itself $k$ times, for some positive integer $k$;
we denote this surface by $kS$.
An important special case is the surface $2S$, which we call the \textbf{double} of $S$.

As we will see in Sections~\ref{subsec:handlebody} and ~\ref{subsec:edgeSeifert},
it is often sufficient to focus our attention on the normal surfaces that
arise as minimal scalar multiples of vertices of the polytope $\mathcal{P}(\mathcal{T})$
(by ``minimal'', we mean the smallest scalar multiple such that the coordinates are all integers);
such surfaces are called \textbf{vertex normal surfaces}\footnote{Some authors use different terminology.
For instance, in~\cite{JacoTollefson1995}, vertex normal surfaces are called ``vertex solutions'',
and the term ``vertex surface'' refers to a connected \emph{two-sided} normal surface
that is either a vertex solution or the double of a vertex solution.}.
Roughly, the significance of vertex normal surfaces is that
they form a finite and algorithmically enumerable ``basis'' for the set of all normal surfaces.

\subsection{Lens spaces}\label{subsec:lens}

We now give a brief review of lens spaces, which will also allow us to specify the conventions that we use.
There are several equivalent ways to define lens spaces.
The most convenient way for our purposes is to construct them by gluing together two solid tori.

We can perform any such gluing by first attaching a thickened disc along
a non-trivial circle $\gamma$ in the boundary of one of the solid tori.
To complete the gluing, we just need to fill the resulting $2$-sphere boundary with a $3$-ball.
Since there is only one way to fill with a $3$-ball, the gluing is completely determined by the choice of the circle $\gamma$.

Thus, we can parametrise all possible lens spaces by parametrising all possible choices for $\gamma$.
To this end, view each solid torus as a product $D^2\times S^1$.
Up to isotopy, every non-trivial circle in the boundary torus $\partial D^2\times S^1$
lifts to a straight line of either rational or infinite slope in the universal cover $\mathbb{R}^2$,
and we can parametrise such circles by this slope;
our convention here will be that a circle of the form $\{x\}\times S^1$ has slope $\infty$,
while a circle of the form $\partial D^2\times\{y\}$ has slope $0$.

With all this in mind, define the \textbf{lens space} $L_{p,q}$ to be the $3$-manifold given by
taking the slope of $\gamma$ to be $\frac{p}{q}$ in the above construction.\footnote{
With the conventions we have chosen here, this is equivalent to saying that
$L_{p,q}$ is obtained by $\frac{p}{q}$ Dehn surgery on the unknot.}
This parametrisation is not unique: it turns out that $L_{p,q}$ and $L_{p',q'}$ are
homeomorphic if and only if $|p'|=|p|$ and $q'\equiv\pm q^{\pm1}\pmod{p}$\cite{Brody1960}.
Some authors exclude degenerate cases---most commonly, the $3$-sphere ($L_{1,0}$)
and/or $S^2\times S^1$ ($L_{0,1}$)---from their definition of lens spaces;
we make no such restrictions here.

\subsection{Seifert fibre spaces}\label{subsec:sfs}

We now give a quick review of Seifert fibre spaces,
which will also allow us to establish some notation.
There are many sources for more comprehensive discussion of Seifert fibre spaces;
much of the material here is based on the discussion in Hatcher's $3$-manifold notes~\cite{Hatcher3Mfld},
but with simplifications in places where it is not necessary to discuss the theory in full generality.

We begin by constructing fibred solid tori.
Take the cylinder $D^2\times[0,1]$, and form a solid torus by gluing
the discs $D^2\times\{0\}$ and $D^2\times\{1\}$ together with a $\frac{2\pi p}{q}$ rotation,
where $q$ is a positive integer, and $p$ is any integer coprime with $q$.
This solid torus is fibred by circles in the following way:
\begin{itemize}
\item There is a fibre $\{0\}\times S^1$ given by closing up the segment $\{0\}\times[0,1]$.
\item The remaining fibres are given by joining together $q$ segments of the form $\{x\}\times[0,1]$.
\end{itemize}
We call this solid torus (with this fibration) a \textbf{fibred solid torus}
with \textbf{multiplicity} $q$.

A \textbf{Seifert fibre space} is a $3$-manifold that is fibred by circles in such a way that
each circle fibre---called a \textbf{Seifert fibre}---has a small solid torus neighbourhood that
forms a fibred solid torus\footnote{Some authors
also allow fibred solid Klein bottle neighbourhoods (for instance, see~\cite{Scott1983}).}.
If the multiplicity of the fibred solid torus neighbourhood is one,
then we call the Seifert fibre a \textbf{regular fibre};
otherwise, we call the Seifert fibre an \textbf{exceptional fibre}\footnote{
Other authors use various other names for exceptional fibres, such as
\emph{multiple fibres}~\cite{Hatcher3Mfld} or \emph{critical fibres}~\cite{Scott1983}.}.
The exceptional fibres are isolated, and they are disjoint from the boundary of the manifold.

Consider a compact Seifert fibre space $\mathcal{M}$,
and let $k$ be the (necessarily finite) number of exceptional fibres in $\mathcal{M}$.
By removing a small solid torus neighbourhood of each exceptional fibre,
we obtain a circle bundle $\mathcal{M}'$ over a surface $B'$;
the surface $B'$ has $k$ boundary components corresponding to the exceptional fibres that we removed
(and possibly some additional boundary components corresponding to boundary components of $\mathcal{M}$).
Let $B$ be the surface obtained by filling these $k$ boundary components of $B'$ with discs;
we say that $\mathcal{M}$ is \textbf{(Seifert) fibred over} $B$, and we call $B$ the \textbf{base surface}.

From this point onwards, we restrict our attention to Seifert fibre spaces that are compact, connected, orientable and irreducible.
The irreducibility restriction only rules out three examples:
\begin{itemize}
\item $S^2\times S^1$ (product of the $2$-sphere $S^2$ and the circle $S^1$),
\item $S^2\simtimes S^1$ (twisted product of $S^2$ and $S^1$), and
\item $\mathbb{R}P^3\ConnSum\mathbb{R}P^3$ (connected sum of two copies of real projective space $\mathbb{R}P^3$).
\end{itemize}

We now say some words about the incompressible, $\partial$-incompressible surfaces\footnote{
In his $3$-manifold notes, Hatcher calls these surfaces ``essential''~\cite[p.~18]{Hatcher3Mfld}.
We do not use this terminology because, as defined in Section~\ref{subsec:embSurfs},
we require that essential surfaces are not boundary parallel.}
in a Seifert fibre space $\mathcal{M}$ (with the restrictions just mentioned).
We call an embedded surface in $\mathcal{M}$ \textbf{vertical} if it consists of a union of regular fibres,
and \textbf{horizontal} if it meets all fibres transversely.
Up to isotopy, every incompressible, $\partial$-incompressible surface in $\mathcal{M}$ is either vertical or horizontal.
In the other direction, every connected two-sided horizontal surface is incompressible and $\partial$-incompressible.
However, this is not quite true for vertical surfaces:
\begin{itemize}
\item if a vertical torus bounds a fibred solid torus, then it is compressible; and
\item if a vertical annulus cuts off from $\mathcal{M}$ a solid torus with the product fibration, then it is $\partial$-compressible.
\end{itemize}
It turns out that these are the only exceptions;
every other connected two-sided vertical surface is incompressible and $\partial$-incompressible.

Suppose, in addition to the restrictions mentioned earlier, that the Seifert fibre space $\mathcal{M}$ is closed.
As mentioned in Section~\ref{subsec:conj}, we are particularly interested in the case where $\mathcal{M}$ is \emph{small}
(recall from Section~\ref{subsec:embSurfs} that this means that $\mathcal{M}$ contains
no two-sided incompressible surfaces other than $S^2$ or $\mathbb{R}P^2$).\footnote{
In the specific context of Seifert fibre spaces, some authors take ``small'' to mean that
we have base surface $S^2$ and at most three exceptional fibres
(for instance, see~\cite{Li2006}).}
It turns out that we can always find an incompressible vertical torus in $\mathcal{M}$ unless either:
\begin{itemize}
\item $\mathcal{M}$ is fibred over $S^2$ with at most three exceptional fibres; or
\item $\mathcal{M}$ is fibred over $\mathbb{R}P^2$ with at most one exceptional fibre.
\end{itemize}
In other words, if $\mathcal{M}$ is small, then it must be fibred in one of these two ways
(but the converse is not true, because a $3$-manifold that is fibred in this way
can still contain two-sided incompressible \emph{horizontal} surfaces).

Consider the case where $\mathcal{M}$ is fibred over $\mathbb{R}P^2$ with at most one exceptional fibre.
If there are no exceptional fibres, then $\mathcal{M}$ must actually be homeomorphic to $\mathbb{R}P^3 \ConnSum \mathbb{R}P^3$;
this is one of the reducible cases that we chose to ignore earlier.
If there is one exceptional fibre (still with base surface $\mathbb{R}P^2$), then $\mathcal{M}$ is a \textbf{prism manifold}\footnote{
This terminology is consistent with~\cite{Scott1983}.
Other authors define prism manifolds differently (for instance, see~\cite{LackenbySchleimer2022}).}.
For various reasons, prism manifolds are often treated as a special case in the literature.
One of the reasons is that prism manifolds also admit a fibration over $S^2$ with three exceptional fibres (we will say more about this shortly).
Since we are looking for triangulations with no Seifert fibre edges,
the existence of multiple Seifert fibrations is a little inconvenient for our purposes,
so we will usually exclude prism manifolds from our discussions.

We now consider the case where $\mathcal{M}$ is fibred over $S^2$ with at most \emph{two} exceptional fibres.
In this case, $\mathcal{M}$ is a lens space, and it has \emph{infinitely} many Seifert fibrations.
For this reason, we will also usually exclude lens spaces when we are working with small Seifert fibre spaces.

This leaves the case where $\mathcal{M}$ is fibred over $S^2$ with exactly three exceptional fibres.
Apart from the prism manifolds mentioned earlier, these all have a unique Seifert fibration.
The following construction gives a natural way to parametrise these Seifert fibre spaces:

\begin{construction}[Seifert fibre spaces with base surface $S^2$ and three exceptional fibres]\label{cons:sfs}
Take the circle bundle $P\times S^1$, where $P$ is a \textbf{pair of pants} (i.e., a $2$-sphere with three small discs removed),
and glue a solid torus to each of the three torus boundary components of our circle bundle.

We specify the three gluings by their slopes
(similar to how we constructed lens spaces in Section~\ref{subsec:lens}).
To ensure that the slopes are well-defined, fix an orientation of $P\times S^1$.
Letting $b_0$, $b_1$ and $b_2$ denote the circles that form the boundary components of $P$,
fix any particular $i\in\{0,1,2\}$ and consider the torus boundary component $b_i\times S^1$.
Here, we use the convention that a circle of the form $b_i\times\{y\}$ has slope $0$,
and any fibre in $b_i\times S^1$ (which is of the form $\{x\}\times S^1$) has slope $\infty$.
Any gluing of a solid torus $\mathcal{T}$ to $b_i\times S^1$ is completely determined by the slope of
a non-trivial circle along which we attach a meridian disc of $\mathcal{T}$,
so we can parametrise the gluing by this slope.

Letting $\frac{\alpha_i}{\beta_i}\in\mathbb{Q}$ (with $\beta_i>0$ and $\gcd(\alpha_i,\beta_i)=1$)
denote the slope by which we glue a solid torus $\mathcal{T}$ to $b_i\times S^1$,
we can extend the fibration on the circle bundle $P\times S^1$ so that
$\mathcal{T}$ becomes a fibred solid torus of multiplicity $\beta_i$.
Thus, provided we glue all three solid tori via non-integer slopes,
the $3$-manifold that results from this construction will be fibred over $S^2$ with exactly three exceptional fibres;
throughout this paper, we will denote this Seifert fibre space by
$\mathcal{S}\left( \frac{\alpha_0}{\beta_0}, \frac{\alpha_1}{\beta_1}, \frac{\alpha_2}{\beta_2} \right)$.
\end{construction}

The parametrisation given at the end of Construction~\ref{cons:sfs} is not unique;
see Proposition~2.1 of~\cite{Hatcher3Mfld} for a precise description of when
two such parametrisations describe the same Seifert fibration.

We also mentioned earlier that it is possible for a $3$-manifold to be Seifert fibred in more than one way.
In particular, we noted that for base surface $S^2$ and three exceptional fibres, such non-uniqueness only occurs for the prism manifolds.
We can now describe this non-uniqueness more precisely:
in addition to being fibred over $\mathbb{R}P^2$ with one exceptional fibre, each prism manifold also admits a Seifert fibration of the form
$\mathcal{S}\left( \frac{1}{2}, -\frac{1}{2}, \frac{\alpha}{\beta} \right)$.

\section{Key tools}\label{sec:tools}

This section discusses some new algorithms and implementations that were crucial for our work.
We anticipate that these algorithms may be useful for other purposes.

\subsection{Handlebody recognition}\label{subsec:handlebody}

To find counterexamples to Conjectures~\ref{conj:core} and~\ref{conj:tunnel},
we need to have algorithms to recognise solid tori and genus-2 handlebodies.
\Regina\ includes an implementation of solid torus
recognition~\cite{Haken1961,JacoTollefson1995} that is
remarkably efficient in practice~\cite{Regina,Burton2013Regina,BurtonOzlen12}.
We generalise this to recognise handlebodies of arbitrary genus,
improving the earlier algorithm of Jaco and Tollefson~\cite[Algorithm~9.3]{JacoTollefson1995}.

The following theorem---which is an immediate consequence of Corollary~6.4 of~\cite{JacoTollefson1995},
due to Jaco and Tollefson---lies at the heart of our handlebody recognition algorithm:

\begin{theorem}\label{thm:compressionDisc}
Let $\mathcal{T}$ be a triangulation of a compact irreducible $3$-manifold with compressible boundary.
Then $\mathcal{T}$ contains a vertex normal surface that realises an essential compression disc.
\end{theorem}

Note that handlebodies satisfy the conditions for this theorem.
With this in mind, the basic strategy of our algorithm is to repeatedly
cut along vertex normal compression discs;
a triangulation $\mathcal{T}$ is a handlebody if and only if
doing this eventually decomposes $\mathcal{T}$ into a collection of $3$-balls.
For computational reasons, we modify this approach in two ways.

First, we expand our focus beyond discs, to include any vertex normal surface $S$ whose Euler characteristic $\chi(S)$ is positive.
The rationale is that $\chi(S)$ can be expressed as a linear function in normal coordinates,
which allows us to exploit linear programming techniques.
This leads to an algorithm for detecting non-trivial normal spheres and discs that is fast in practice
(even though, in theory, this is actually the exponential-time bottleneck for
both solid torus recognition and handlebody recognition);
see~\cite{BurtonOzlen12} for details.

Second, because cutting along a normal surface can increase the number of tetrahedra exponentially,
we instead use a technique called \textbf{crushing}.
Crushing was first developed by Jaco and Rubinstein~\cite{JacoRubinstein2003},
following earlier unpublished work of Casson,
and was later refined by the first author~\cite{Burton2014}.
Crushing a non-trivial normal surface has the following crucial benefit:
it is guaranteed to produce a new triangulation with strictly fewer tetrahedra than before.
The trade-off is that crushing a normal sphere or disc could have topological side-effects,
as detailed in the following theorem (this is a restatement of Theorem~2 of~\cite{Burton2014},
which is itself a summary of a series of results from~\cite{JacoRubinstein2003}):

\begin{theorem}\label{thm:crushing}
Let $\mathcal{T}$ be a $3$-manifold triangulation whose underlying $3$-manifold $\mathcal{M}$ is orientable.
Let $S$ be a normal $2$-sphere or disc in $\mathcal{T}$.
Crushing $S$ yields a $3$-manifold triangulation $\mathcal{T}'$ whose underlying $3$-manifold $\mathcal{M}'$ is
obtained from $\mathcal{M}$ by a finite (and possibly empty) sequence of the following operations:
\begin{itemize}[nosep]
\item undoing connected sums;
\item cutting along properly embedded discs;
\item filling boundary $2$-spheres with $3$-balls; and
\item deleting $3$-ball, $3$-sphere, $\mathbb{R}P^3$, $L_{3,1}$ or $S^2\times S^1$ components.
\end{itemize}
\end{theorem}

For the rest of this section, call a connected orientable $3$-manifold \textbf{interesting} if:
\begin{itemize}
\item it has exactly one boundary component; and
\item its first homology is $\mathbb{Z}^g$, where $g$ is the genus of the boundary surface.
\end{itemize}
In other words, we consider a $3$-manifold $\mathcal{M}$ to be interesting if it
has the same basic properties as a handlebody, which means that we need
more sophisticated techniques to decide whether $\mathcal{M}$ is a handlebody.
In analogy with homology spheres, one could also say that an interesting $3$-manifold is a ``homology handlebody''.
Our handlebody recognition algorithm relies on the following consequence of Theorem~\ref{thm:crushing}:

\begin{proposition}\label{prop:crushHandlebody}
Let $\mathcal{T}$ be a bounded triangulation of an interesting $3$-manifold,
and suppose that $\mathcal{T}$ contains a non-trivial normal $2$-sphere or disc $S$.
Let $\mathcal{T}'$ denote the triangulation obtained by crushing $S$.
Each component of $\mathcal{T}'$ is either:
\begin{itemize}[nosep]
\item a closed $3$-manifold with trivial first homology; or
\item (a bounded triangulation of) an interesting $3$-manifold.
\end{itemize}
Moreover, $\mathcal{T}$ is a handlebody if and only if
every component of $\mathcal{T}'$ is either a $3$-sphere or a handlebody.
\end{proposition}

\begin{proof}
The underlying $3$-manifold of $\mathcal{T}'$ is obtained
from the underlying $3$-manifold of $\mathcal{T}$ by performing
some sequence of the operations listed in Theorem~\ref{thm:crushing}.
We prove the proposition by induction on the number of such operations.

For the base case, if we do not perform any of the operations from Theorem~\ref{thm:crushing},
then the underlying $3$-manifolds for $\mathcal{T}$ and $\mathcal{T}'$ are homeomorphic.
In this case, there is nothing to prove.

For the inductive step, suppose we alter the underlying $3$-manifold of $\mathcal{T}$ using some sequence of operations from Theorem~\ref{thm:crushing}.
Let $\mathcal{M}$ denote any particular component of the resulting $3$-manifold,
and let $\mathcal{M}'$ denote the $3$-manifold obtained from $\mathcal{M}$ by performing any single operation from Theorem~\ref{thm:crushing}.
It suffices to prove the following two claims:
\begin{enumerate}[label={(\alph*)}]
\item\label{handlebody:closedOrInteresting}
Each component of $\mathcal{M}'$ is either: a closed $3$-manifold with trivial first homology, or an interesting $3$-manifold.
\item\label{handlebody:sphereOrHandlebody}
The $3$-manifold $\mathcal{M}$ is either a $3$-sphere or a handlebody if and
only if every component of $\mathcal{M}'$ is either a $3$-sphere or a handlebody.
\end{enumerate}
The justification for these two claims depends on whether $\mathcal{M}$ is closed or interesting,
and on which operation from Theorem~\ref{thm:crushing} relates $\mathcal{M}$ and $\mathcal{M}'$.

We begin with the case where $\mathcal{M}$ is closed, since this case is comparatively straightforward.
Only two of the operations from Theorem~\ref{thm:crushing} apply here:
\begin{itemize}
\item We could delete $\mathcal{M}$ entirely, in which case $\mathcal{M}'$ is empty.
This only occurs if $\mathcal{M}$ is a copy of $S^3$, $L_{3,1}$ or $S^2\times S^1$;
however, as part of the inductive hypothesis, we may assume that $\mathcal{M}$ has trivial first homology, so $\mathcal{M}$ must in fact be homeomorphic to $S^3$.
Thus, in this case, we see that claims~\ref{handlebody:closedOrInteresting} and~\ref{handlebody:sphereOrHandlebody} both hold trivially.
\item We could undo a connected sum, in which case $\mathcal{M}'$ consists of two closed components $\mathcal{C}_1$ and $\mathcal{C}_2$.
Since
\[
H_1(\mathcal{C}_1) \oplus H_1(\mathcal{C}_2) \simeq H_1(\mathcal{C}_1\ConnSum\mathcal{C}_2) \simeq H_1(\mathcal{M}) \simeq 0,
\]
we conclude that $\mathcal{C}_1$ and $\mathcal{C}_2$ both have trivial first homology,
and hence that claim~\ref{handlebody:closedOrInteresting} holds.
Claim~\ref{handlebody:sphereOrHandlebody} also holds, since $\mathcal{M}$ is a $3$-sphere
if and only if $\mathcal{C}_1$ and $\mathcal{C}_2$ are both $3$-spheres.
\end{itemize}

With that out of the way, we devote the rest of this proof to the case where $\mathcal{M}$ is interesting;
let $g$ denote the genus of the boundary surface $\partial\mathcal{M}$.
Of the operations listed in Theorem~\ref{thm:crushing}, the following are relatively straightforward to handle:
\begin{itemize}
\item We could delete $\mathcal{M}$ entirely, in which case $\mathcal{M}'$ is empty;
this only occurs if $\mathcal{M}$ is a $3$-ball.
In this case, claims~\ref{handlebody:closedOrInteresting} and~\ref{handlebody:sphereOrHandlebody} both hold trivially.
\item If $\mathcal{M}$ has $2$-sphere boundary, then we could fill this $2$-sphere with a $3$-ball, in which case $\mathcal{M}'$ is closed.
Since $\mathcal{M}$ is interesting and has genus-$0$ boundary, we have $H_1(\mathcal{M})=0$, which implies that $\mathcal{M}'$ also has trivial first homology;
thus, claim~\ref{handlebody:closedOrInteresting} holds.
Moreover, $\mathcal{M}$ is a $3$-ball if and only if $\mathcal{M}'$ is a $3$-sphere, so claim~\ref{handlebody:sphereOrHandlebody} also holds.
\item We could undo a connected sum.
In this case, since $\mathcal{M}$ has exactly one boundary component, $\mathcal{M}'$ has two components:
a closed component $\mathcal{C}$, and a component $\mathcal{B}$ with the same boundary as $\mathcal{M}$.
Since $H_1(\partial\mathcal{B})$ has rank $2g$, we know that $H_1(\mathcal{B})$ has rank at least $g$ (see Lemma~3.5 of~\cite{Hatcher3Mfld}).
Thus, since
\[
H_1(\mathcal{B})\oplus H_1(\mathcal{C})\simeq H_1(\mathcal{B}\ConnSum\mathcal{C})\simeq H_1(\mathcal{M})\simeq \mathbb{Z}^g,
\]
we conclude that $H_1(\mathcal{B})\simeq \mathbb{Z}^g$ and $H_1(\mathcal{C})\simeq0$,
and hence that claim~\ref{handlebody:closedOrInteresting} holds.
Moreover, $\mathcal{M}$ is a handlebody if and only if $\mathcal{C}$ is a $3$-sphere and $\mathcal{B}$ is a handlebody,
so claim~\ref{handlebody:sphereOrHandlebody} also holds.\footnote{
Here, it is crucial that $\mathcal{M}$ has exactly one boundary component,
as this rules out the case where $\mathcal{M}$ is a connected sum of two handlebodies
(otherwise, it would be possible for $\mathcal{M}'$ to be the disjoint union of
these two handlebodies, even though $\mathcal{M}$ is not itself a handlebody).
Later in the proof, we will see why it is necessary to make the
stronger assumption that $\mathcal{M}$ is actually interesting.}
\end{itemize}
It remains to consider what happens if we cut along a properly embedded disc $D$;
the argument for this remaining operation requires a little more effort than the other operations.

We first consider the case where $D$ is a \emph{separating} disc.
In this case, $\mathcal{M}'$ has two components $\mathcal{B}_1$ and $\mathcal{B}_2$,
each with a single boundary component.
For each $i\in\{1,2\}$, let $g_i$ denote the genus of the boundary surface $\partial\mathcal{B}_i$;
note that $H_1(\mathcal{B}_i)$ has rank at least $g_i$.
By retracting $D$ to a point, we see that $\mathcal{M}$ is homotopy-equivalent to
the wedge sum $\mathcal{B}_1\vee\mathcal{B}_2$, and hence that
\[
H_1(\mathcal{B}_1)\oplus H_1(\mathcal{B}_2)\simeq
H_1(\mathcal{B}_1\vee\mathcal{B}_2)\simeq H_1(\mathcal{M})\simeq \mathbb{Z}^g.
\]
From this, together with the fact that $g=g_1+g_2$,
we conclude that $\mathcal{B}_1$ and $\mathcal{B}_2$ are both interesting, and hence that claim~\ref{handlebody:closedOrInteresting} holds.
Claim~\ref{handlebody:sphereOrHandlebody} also holds, since $\mathcal{M}$ is a handlebody if and only if
$\mathcal{B}_1$ and $\mathcal{B}_2$ are both handlebodies.

All that remains is to consider the case where $D$ is \emph{non-separating},
in which case $\mathcal{M}'$ consists of a single bounded component;
this case requires much more work than all the preceding cases.
We begin by noting that $\mathcal{M}$ is a handlebody if and only if $\mathcal{M}'$ is a handlebody,
and hence that claim~\ref{handlebody:sphereOrHandlebody} holds.
Thus, the rest of this proof is entirely devoted to showing that $\mathcal{M}'$ is interesting,
and hence verifying claim~\ref{handlebody:closedOrInteresting}.
The first step is to show that $\partial D$ is a non-separating curve on the boundary surface $\partial\mathcal{M}$,
as this will imply that $\mathcal{M}'$ has a single boundary component.\footnote{
Earlier, we pointed out that it is crucial that we never
create a component with more than one boundary component.
As we will see shortly, proving this in the present case requires the assumption that
$H_1(\mathcal{M})\simeq \mathbb{Z}^g$, which is why we need to assume that $\mathcal{M}$ is interesting.}

Suppose instead that $\partial D$ separates $\partial\mathcal{M}$ into two pieces $P$ and $Q$.
Let $h$ denote the genus of $P$.
Since $D$ is a non-separating disc, there is a closed loop $\gamma$ in
the interior of $\mathcal{M}$ that meets $D$ exactly once.
Push $P\cup D$ slightly off the boundary to obtain a closed genus-$h$ surface $E$ that meets $\gamma$ exactly once.
Let $X'$ be a small neighbourhood of $E\cup\gamma$, and let $Y'=\mathcal{M}-X'$.
The boundary of $X'$ is a surface of genus $2h$;
let $N$ be a small neighbourhood of this surface.
Taking $X=X'\cup N$ and $Y=Y'\cup N$ gives a pair of subspaces whose interiors together cover $\mathcal{M}$.
Noting that $X\cap Y=N$,
the Mayer-Vietoris sequence for $X$ and $Y$ gives us the following exact sequence:
\begin{center}
\begin{tikzcd}[row sep=tiny,
ar symbol/.style = {draw=none,"\textstyle#1" description,sloped},
isom/.style = {ar symbol={\simeq}},]
\cdots \ar[r] & H_1(N) \ar[r,"\varphi"] \ar[d,isom] & H_1(X)\oplus H_1(Y) \ar[r,"\psi"] \ar[d,isom] &
H_1(\mathcal{M}) \ar[r] \ar[d,isom] & \cdots \\
& \mathbb{Z}^{4h} & \mathbb{Z}^{2h+1}\oplus H_1(Y) & \mathbb{Z}^g
\end{tikzcd}
\end{center}
Moreover, since $Y$ has two boundary components, one of genus $g$ and the other of genus $2h$,
we know that $H_1(Y)$ has rank at least $g+2h$.
Thus, in total, $H_1(X)\oplus H_1(Y)$ has rank at least $g+4h+1$.
But we also know that $\image(\psi)$ is free abelian with rank at most $g$,
so exactness tells us that $\image(\varphi)$ must have rank at least $4h+1$.
This is impossible, since $H_1(N)$ only has rank $4h$.

Thus, we obtain the desired conclusion that $\partial D$ is a non-separating curve on $\partial\mathcal{M}$.
This implies that $\mathcal{M}'$ has exactly one boundary component,
and that this boundary component has genus $g-1$.
To complete the proof, we just need to show that $\mathcal{M}'$ is interesting by
verifying that $H_1(\mathcal{M}')\simeq \mathbb{Z}^{g-1}$.

One way to do this is to examine the Mayer-Vietoris sequence for subspaces $X$ and $Y$ constructed as follows:
\begin{itemize}
\item Let $X$ be a subspace of $\mathcal{M}$ homeomorphic to $\mathcal{M}'$.
\item Let $Y$ be a $3$-ball neighbourhood of the disc $D$.
\end{itemize}
We can choose $X$ and $Y$ so that their interiors together cover $\mathcal{M}$,
and the intersection $N:=X\cap Y$ is a disjoint pair of $3$-balls.
This gives the following exact sequence:
\begin{center}
\begin{tikzcd}[column sep=small, row sep=tiny,
ar symbol/.style = {draw=none,"\textstyle#1" description,sloped},
isom/.style = {ar symbol={\simeq}},]
\cdots \ar[r] & H_1(N) \ar[r,"\varphi_1"] \ar[d,isom] & H_1(X)\oplus H_1(Y) \ar[r,"\psi_1"] \ar[d,isom] &
H_1(\mathcal{M}) \ar[r,"\partial"] \ar[d,isom] & H_0(N) \ar[r,"\varphi_0"] \ar[d,isom] &
H_0(X)\oplus H_0(Y) \ar[r,"\psi_0"] \ar[d,isom] & H_0(\mathcal{M}) \ar[r] \ar[d,isom] & 0 \\
& 0 & H_1(\mathcal{M}') & \mathbb{Z}^g & \mathbb{Z}^2 & \mathbb{Z}^2 & \mathbb{Z}
\end{tikzcd}
\end{center}
By exactness, we know that:
\begin{itemize}
\item $H_1(\mathcal{M}')\simeq \ker(\partial)$; and
\item $\image(\psi_0)$, $\image(\varphi_0)$ and $\image(\partial)$ all have rank $1$.
\end{itemize}
Thus, we see that $H_1(\mathcal{M}')\simeq \mathbb{Z}^{g-1}$, as required.
\end{proof}

The last two things that we rely on are algorithms to recognise the $3$-sphere~\cite{JacoRubinstein2003,Rubinstein1995,Thompson1994} and the $3$-ball
($3$-ball recognition is a well-known variant of $3$-sphere recognition).
Implementations for both algorithms are available in \Regina~\cite{Burton2013Regina,Regina}.

The following handlebody recognition algorithm is now also available in \Regina~\cite{Regina}:

\begin{algorithm}[Handlebody recognition]\label{algm:handlebody}
To test whether a bounded triangulation $\mathcal{T}$ is a handlebody, and if so determine its genus:
\begin{enumerate}[nosep, label={(\arabic*)}]
\item\label{item:handlebodyBasic} Check that $\mathcal{T}$ is
connected, orientable, and has exactly one boundary component.
If $\mathcal{T}$ fails to satisfy any of these conditions, then terminate and return $-1$ (indicating that $\mathcal{T}$ is not a handlebody).
\item\label{item:handlebodyHomology} Compute the genus $g$ of the boundary of $\mathcal{T}$.
	\begin{itemize}[nosep]
	\item In the case that $g=0$, check whether $\mathcal{T}$ is a $3$-ball.
	If it is, terminate and return $0$ (indicating that $\mathcal{T}$ is a genus-$0$ handlebody);
	otherwise, terminate and return $-1$.
	\item In the case that $g>0$, check whether the first homology of $\mathcal{T}$ is $\mathbb{Z}^g$.
	If it is not, terminate and return $-1$.
	\end{itemize}
\item\label{item:handlebodyLoop} Create a list $\mathcal{L}$ of triangulations to process,
which initially contains $\mathcal{T}$ (and nothing else).
While $\mathcal{L}$ is non-empty:
	\begin{enumerate}[nosep, label={(\roman*)}]
	\item\label{item:beginLoop} Let $\mathcal{F}$ be the first triangulation that appears in $\mathcal{L}$.
	Remove $\mathcal{F}$ from $\mathcal{L}$.
	\item Find a non-trivial normal $2$-sphere or disc $S$ in $\mathcal{F}$.
	If no such surface exists, then terminate and return $-1$.
	\item Crush $S$. For each component $\mathcal{C}$ of the triangulation $\mathcal{R}$ that results from crushing:
		\begin{itemize}[nosep]
		\item If $\mathcal{C}$ is closed, check whether $\mathcal{C}$ is a $3$-sphere.
		If it is, discard $\mathcal{C}$ and move on to the next component of $\mathcal{R}$;
		otherwise, terminate and return $-1$.
		\item If $\mathcal{C}$ has $2$-sphere boundary, check whether $\mathcal{C}$ is a $3$-ball.
		If it is, discard $\mathcal{C}$ and move on to the next component of $\mathcal{R}$;
		otherwise, terminate and return $-1$.
		\item In any other case, add $\mathcal{C}$ to the list $\mathcal{L}$,
		and move on to the next component of $\mathcal{R}$.
		\end{itemize}
	\end{enumerate}
\item Once there are no more triangulations in $\mathcal{L}$, terminate and return $g$ (indicating that $\mathcal{T}$ is a genus-$g$ handlebody).
\end{enumerate}
\end{algorithm}

\begin{theorem}[Correctness]\label{thm:handlebody}
Algorithm~\ref{algm:handlebody} correctly determines whether a given bounded triangulation is a handlebody.
\end{theorem}

\begin{proof}
Step~\ref{item:handlebodyHomology} eliminates $3$-manifolds with $2$-sphere boundary,
and $3$-manifolds that are not interesting.
Thus, throughout this proof, we can safely assume that the initial triangulation $\mathcal{T}$
represents an interesting $3$-manifold whose boundary surface has positive genus.

We begin by proving that the algorithm terminates.
It suffices to show that the loop in step~\ref{item:handlebodyLoop}
does not continue indefinitely.
Indeed, each time we remove a triangulation $\mathcal{F}$ from the list $\mathcal{L}$,
we either crush a non-trivial normal surface $S$,
or we terminate because there is no suitable surface $S$.
Crushing $S$ is guaranteed to reduce the number of tetrahedra,
so even if we add some triangulations back into the list $\mathcal{L}$,
the total number of tetrahedra in this list must decrease.
This number cannot decrease indefinitely, so the algorithm must terminate.

We now prove that the algorithm correctly determines whether the input triangulation $\mathcal{T}$ is a handlebody.
The key is the following invariant, which is preserved every time we
make a full pass through the loop in step~\ref{item:handlebodyLoop}:
every triangulation in $\mathcal{L}$ is interesting,
and $\mathcal{T}$ is a handlebody with positive genus if and only if
every triangulation in $\mathcal{L}$ is actually a handlebody with positive genus.

To prove that this invariant is indeed preserved, we first consider the case where
the triangulation $\mathcal{F}$ that we remove from $\mathcal{L}$ is a handlebody with positive genus.
In this case, Theorem~\ref{thm:compressionDisc} guarantees that
we will find a non-trivial normal $2$-sphere or disc $S$ in $\mathcal{F}$.
Then, by Proposition~\ref{prop:crushHandlebody}, crushing $S$ yields
a new triangulation whose components are all either $3$-spheres or handlebodies.
We discard all the $3$-spheres and $3$-balls, leaving only handlebodies with positive genus.
This suffices to show that the invariant is always preserved when
the initial triangulation $\mathcal{T}$ is a handlebody.

Suppose now that we remove a triangulation $\mathcal{F}$ from $\mathcal{L}$
that is not a handlebody.
We need to check that the invariant is preserved if we pass through the loop
without terminating and returning $-1$.
So, we can assume that we crush a non-trivial normal $2$-sphere or disc in $\mathcal{F}$,
yielding a new triangulation $\mathcal{F}'$.
We can also assume that all the closed components of $\mathcal{F}'$ are $3$-spheres,
and that all the components of $\mathcal{F}'$ with $2$-sphere boundary are $3$-balls;
we discard all such components.
Proposition~\ref{prop:crushHandlebody} tells us that the remaining components are all interesting,
and that at least one of these remaining components is not a handlebody;
we add all of these interesting components back into $\mathcal{L}$,
thus ensuring that $\mathcal{L}$ still contains at least one triangulation that is not a handlebody.
This shows that the invariant is always preserved when
the initial triangulation $\mathcal{T}$ is not a handlebody.

Altogether, we see that if $\mathcal{T}$ \emph{is} a handlebody,
then the list $\mathcal{L}$ must eventually become empty,
in which case we terminate and return the genus $g$ computed in step~\ref{item:handlebodyHomology};
this correctly indicates that $\mathcal{T}$ is a handlebody of genus $g$.
On the other hand, if $\mathcal{T}$ is \emph{not} a handlebody, then $\mathcal{L}$ can never become empty,
so at some point we must terminate and correctly return $-1$.
\end{proof}

\subsection{Detecting Seifert fibre edges}\label{subsec:edgeSeifert}

To find counterexamples to Conjecture~\ref{conj:fibre},
we need to be able to take an edge $e$ in a one-vertex triangulation of a small Seifert fibre space,
and determine whether $e$ is a Seifert fibre edge.
Since this is not easy to test conclusively, and since our main goal is to
find a triangulation that definitely has no Seifert fibre edges,
we use an algorithm (see Algorithm~\ref{algm:seifert}) that only gives conclusive answers in one direction:
certifying that an edge is \emph{not} isotopic to a Seifert fibre.

Throughout this section, let $\mathcal{T}$ be a one-vertex triangulation of
an irreducible small Seifert fibre space that is neither a lens space nor a prism manifold,
and fix an edge $e$ of $\mathcal{T}$.
(Recall from our discussion in Section~\ref{subsec:sfs} that $\mathcal{T}$ must be
fibred over $S^2$ with three exceptional fibres, and that this is the unique Seifert fibration on $\mathcal{T}$.)
The first step of our algorithm is to build a triangulation $\mathcal{D}$ of
the $3$-manifold obtained by deleting a small solid torus neighbourhood of $e$.
From here, our strategy is to take a list of normal surfaces that
are guaranteed to occur in $\mathcal{D}$ if $e$ is a Seifert fibre edge,
and check whether all such surfaces occur in $\mathcal{D}$;
if any of these surfaces are missing, then we know for sure that $e$ is not a Seifert fibre edge.
This strategy appears to be effective as long as we use a sufficiently exhaustive list of surfaces.

The key theoretical result that we use is the following,
which is a consequence of Corollary~6.8 of~\cite{JacoTollefson1995},
due to Jaco and Tollefson:

\begin{proposition}\label{prop:essAnnTor}
Let $\mathcal{F}$ be a bounded triangulation of an orientable, irreducible and $\partial$-irreducible Seifert fibre space.
Call a surface in $\mathcal{F}$ \textbf{good} if it is essential,
and is either a vertex normal surface or the double of a vertex normal surface.
\begin{itemize}[nosep]
\item If $\mathcal{F}$ has a vertical essential annulus, then it has a good vertical annulus.
\item If $\mathcal{F}$ has a vertical essential torus, then it has a good vertical torus.
Moreover, if $\mathcal{F}$ has more than one vertical essential torus (up to isotopy),
then it has at least two good vertical tori.
\end{itemize}
\end{proposition}

\begin{proof}
First, suppose $\mathcal{F}$ has a vertical essential annulus $A$.
Corollary~6.8 of~\cite{JacoTollefson1995} tells us that $A$ is isotopic to
a normal annulus $S$ such that for some $k\geqslant1$, the surface $kS$ is a sum of good annuli and good tori.
Observe that at least one of the summands must be a good vertical annulus.

Suppose now that $\mathcal{F}$ has a vertical essential torus $A$.
This time, Corollary~6.8 of \cite{JacoTollefson1995} tells us that $A$ is isotopic to
a normal torus $S$ such that for some $k\geqslant1$, the surface $kS$ is a sum of good tori.
Since $\mathcal{F}$ has non-empty boundary, these good torus summands must all be vertical.
Moreover, if $\mathcal{F}$ only contains one good vertical torus $V$, then sums of $V$ with itself can only produce multiple parallel copies of $V$;
in this case, $V$ must therefore be (up to isotopy) the only vertical essential torus in $\mathcal{F}$.
\end{proof}

We first discuss how to test whether $e$ is isotopic to an exceptional fibre,
since this is much easier than the regular case.
If $e$ \emph{is} isotopic to an exceptional fibre, then observe that the triangulation $\mathcal{D}$
is fibred over a disc with two exceptional fibres, which implies that $\mathcal{D}$
contains a vertical annulus $A$ that cuts $\mathcal{D}$ into a pair of solid tori.
Thus, if we can show that no such annulus $A$ exists, then we will have certified
that $e$ is not isotopic to an exceptional fibre.
This is the main idea behind the following algorithm (which is an important subroutine for Algorithm~\ref{algm:seifert}):

\begin{algorithm}\label{algm:cutTwoSolidTori}
Let $\mathcal{F}$ be a bounded triangulation of an orientable irreducible $3$-manifold
with one torus boundary component (and no other boundary components).
To test whether $\mathcal{F}$ is fibred over a disc with two exceptional fibres:
\begin{enumerate}[nosep, label={(\arabic*)}]
\item Create an empty list $\mathcal{A}$.
\item\label{item:excFibEnum}
Enumerate all vertex normal surfaces in $\mathcal{F}$.
For each such surface $S$:
	\begin{itemize}[nosep]
	\item If $S$ is an essential compression disc, terminate and return \False.
	\item If $S$ is an annulus, add $S$ to $\mathcal{A}$.
	\item If $2S$ is an annulus, add $2S$ to $\mathcal{A}$.
	\end{itemize}
\item\label{item:excFibAnnulus} For each annulus $S$ in $\mathcal{A}$, check whether cutting along $S$
decomposes $\mathcal{F}$ into a pair of solid tori.
If there is an annulus that gives such a decomposition, terminate and return \Unknown;
otherwise, terminate and return \False.
\end{enumerate}
\end{algorithm}

\begin{proposition}\label{prop:cutTwoSolidTori}
If Algorithm~\ref{algm:cutTwoSolidTori} returns \False, then the input triangulation $\mathcal{F}$
cannot be fibred over a disc with two exceptional fibres.
\end{proposition}

\begin{proof}
Suppose the input triangulation $\mathcal{F}$ is fibred over a disc with two exceptional fibres.
We need to prove that the algorithm will not return \False.
Since $\mathcal{F}$ is $\partial$-irreducible, the algorithm cannot return \False\ in step~\ref{item:excFibEnum}.
Thus, it suffices to show that the algorithm will find an annulus that cuts $\mathcal{F}$
into a pair of solid tori, and hence return \Unknown\ in step~\ref{item:excFibAnnulus}.
For this, observe that $\mathcal{F}$ has (up to isotopy) exactly one
vertical essential annulus $A$, and that $A$ has exactly the required property;
Proposition~\ref{prop:essAnnTor} then tells us that $A$ appears as either a vertex normal surface or the double of a vertex normal surface,
so the algorithm will find $A$ amongst the normal surfaces that it generates in step~\ref{item:excFibEnum}.
\end{proof}

We now discuss how to test whether $e$ is isotopic to a regular fibre.
For this, we rely on the following observation:

\begin{observation}\label{obs:regFibre}
Suppose $e$ is isotopic to a regular fibre.
In this case, the triangulation $\mathcal{D}$ is fibred over a disc with three exceptional fibres, which has the following implications:
\begin{enumerate}[nosep,label={(\alph*)}]
\item\label{regFibre:annuli}
Up to isotopy, $\mathcal{D}$ contains exactly three vertical essential annuli.
Each such annulus separates one exceptional fibre from the other two,
and therefore cuts $\mathcal{D}$ into two pieces $\mathcal{P}_0$ and $\mathcal{P}_1$,
where $\mathcal{P}_0$ is a solid torus and $\mathcal{P}_1$ is fibred over a disc with two exceptional fibres.
\item\label{regFibre:tori}
Up to isotopy, $\mathcal{D}$ contains exactly three vertical essential tori.
Each such torus encloses two of the three exceptional fibres, and therefore cuts $\mathcal{D}$ into two pieces $\mathcal{Q}_0$ and $\mathcal{Q}_1$,
where $\mathcal{Q}_0$ is fibred over an annulus with one exceptional fibre and $\mathcal{Q}_1$ is fibred over a disc with two exceptional fibres.
\end{enumerate}
\end{observation}

From part~\ref{regFibre:tori} of Observation~\ref{obs:regFibre}, we see that to rule out the existence of vertical essential tori in $\mathcal{D}$,
we need to be able to recognise $3$-manifolds $\mathcal{M}$ that are fibred over an annulus with one exceptional fibre.
For this, we rely on the fact that such a $3$-manifold $\mathcal{M}$ has a vertical essential annulus that cuts $\mathcal{M}$ into a single solid torus.
The following algorithm (which is another important subroutine for Algorithm~\ref{algm:seifert}) proceeds by attempting to rule out the existence of such an annulus:

\begin{algorithm}\label{algm:cutOneSolidTorus}
Let $\mathcal{F}$ be a bounded triangulation of an orientable irreducible $3$-manifold
with two torus boundary components (and no other boundary components).
To test whether $\mathcal{F}$ is fibred over an annulus with one exceptional fibre:
\begin{enumerate}[nosep, label={(\arabic*)}]
\item Create an empty list $\mathcal{A}$.
\item\label{item:cutOneEnum}
Enumerate all vertex normal surfaces in $\mathcal{F}$.
For each such surface $S$:
	\begin{itemize}[nosep]
	\item If $S$ is an essential compression disc, terminate and return \False.
	\item If $S$ is an annulus, add $S$ to $\mathcal{A}$.
	\item If $2S$ is an annulus, add $2S$ to $\mathcal{A}$.
	\end{itemize}
\item\label{item:cutOneAnnulus} For each annulus $S$ in $\mathcal{A}$, check whether cutting along $S$
decomposes $\mathcal{F}$ into a single solid torus.
If there is an annulus that gives such a decomposition, terminate and return \Unknown;
otherwise, terminate and return \False.
\end{enumerate}
\end{algorithm}

\begin{proposition}\label{prop:cutOneSolidTorus}
If Algorithm~\ref{algm:cutOneSolidTorus} returns \False,
then the input triangulation $\mathcal{F}$ cannot be fibred over an annulus with one exceptional fibre.
\end{proposition}

\begin{proof}
Suppose the input triangulation $\mathcal{F}$ is fibred over an annulus with one exceptional fibre.
We need to prove that the algorithm will not return \False;
the proof is almost identical to the proof of Proposition~\ref{prop:cutTwoSolidTori}.
Since $\mathcal{F}$ is $\partial$-irreducible, the algorithm cannot return \False\ in step~\ref{item:cutOneEnum}.
Thus, it suffices to show that the algorithm will find an annulus that cuts $\mathcal{F}$
into a single solid torus, and hence return \Unknown\ in step~\ref{item:cutOneAnnulus}.
For this, observe that $\mathcal{F}$ has (up to isotopy) exactly one vertical essential annulus $A$, and that $A$ has exactly the required property;
Proposition~\ref{prop:essAnnTor} then tells us that $A$ appears as either a vertex normal surface or the double of a vertex normal surface,
so the algorithm will find $A$ amongst the normal surfaces that it generates in step~\ref{item:cutOneEnum}.
\end{proof}

We are finally ready to present our algorithm for testing whether $e$ is a Seifert fibre edge (Algorithm~\ref{algm:seifert}).
We have already mentioned that the difficult case is testing whether $e$ is isotopic to a regular fibre,
and that in this case $\mathcal{D}$ must have several vertical essential annuli and vertical essential tori.
With this in mind, the overall strategy is similar to the previous two algorithms:
use Proposition~\ref{prop:essAnnTor} to try to rule out the existence of these vertical surfaces.

\begin{algorithm}\label{algm:seifert}
Let $\mathcal{T}$ be a one-vertex triangulation of an irreducible small Seifert fibre space that is neither a lens space nor a prism manifold.
To test whether an edge $e$ of $\mathcal{T}$ is a Seifert fibre edge:
\begin{enumerate}[nosep, label={(\arabic*)}]
\item\label{item:drill}
Build a bounded triangulation $\mathcal{D}$ of the $3$-manifold obtained by
deleting a small solid torus neighbourhood of $e$.
\item\label{item:mainLists}
Create empty lists $\mathcal{A}$, $\mathcal{B}$ and $\mathcal{L}$.
\item\label{item:mainEnumerate} Enumerate all vertex normal surfaces in $\mathcal{D}$.
For each such surface $S$:
	\begin{itemize}[nosep]
	\item If $S$ is an essential compression disc, terminate and return \False.
	\item If $S$ is an annulus, add $S$ to $\mathcal{A}$.
	\item If $2S$ is an annulus, add $2S$ to $\mathcal{A}$.
	\item If $S$ is a torus, add $S$ to $\mathcal{B}$.
	\item If $2S$ is a torus, add $2S$ to $\mathcal{B}$.
	\end{itemize}
\item\label{item:emptyAnnulusList} If the list $\mathcal{A}$ is still empty, terminate and return \False.
\item\label{item:cutAnnulus} For each annulus $S$ in $\mathcal{A}$:
	\begin{enumerate}[nosep, label={(\roman*)}]
	\item Cut along $S$, and check whether this yields two components $\mathcal{C}_0$ and $\mathcal{C}_1$.
	If not, move on to the next annulus in $\mathcal{A}$.
	\item Check whether $\mathcal{C}_0$ and $\mathcal{C}_1$ are solid tori.
		\begin{itemize}[nosep]
		\item If they are both solid tori, terminate and return \Unknown.
		\item If $\mathcal{C}_0$ is a solid torus but $\mathcal{C}_1$ is not,
		add $\mathcal{C}_1$ to $\mathcal{L}$.
		\item If $\mathcal{C}_1$ is a solid torus but $\mathcal{C}_0$ is not,
		add $\mathcal{C}_0$ to $\mathcal{L}$.
		\end{itemize}
	\end{enumerate}
\item\label{item:notEnoughTori} If the list $\mathcal{B}$ contains fewer than two tori,
terminate and return \False.
\item\label{item:cutTwoSolidTori} Run Algorithm~\ref{algm:cutTwoSolidTori} on each triangulation in the list $\mathcal{L}$;
if the output is \False\ for \emph{every} such triangulation, then terminate and return \False.
\item Initialise an integer variable $n$ to be $0$.
\item\label{item:cutTori} For each torus $S$ in $\mathcal{B}$:
	\begin{enumerate}[nosep, label={(\roman*)}]
	\item Cut along $S$, and check whether this yields two components $\mathcal{C}_0$ and $\mathcal{C}_1$.
	If not, move on to the next torus in $\mathcal{B}$.
	\item Count the number $b$ of boundary components in $\mathcal{C}_0$.
		\begin{itemize}[nosep]
		\item In the case where $b=1$, check that:
			\begin{itemize}[nosep]
			\item $\mathcal{C}_1$ has $2$ boundary components;
			\item running Algorithm~\ref{algm:cutTwoSolidTori} on $\mathcal{C}_0$
			does not result in output \False; and
			\item running Algorithm~\ref{algm:cutOneSolidTorus} on $\mathcal{C}_1$
			does not result in output \False.
			\end{itemize}
		If these conditions are all satisfied, increase $n$ by $1$.
		\item In the case where $b=2$, check that:
			\begin{itemize}[nosep]
			\item $\mathcal{C}_1$ has $1$ boundary component;
			\item running Algorithm~\ref{algm:cutOneSolidTorus} on $\mathcal{C}_0$
			does not result in output \False; and
			\item running Algorithm~\ref{algm:cutTwoSolidTori} on $\mathcal{C}_1$
			does not result in output \False.
			\end{itemize}
		If these conditions are all satisfied, increase $n$ by $1$.
		\end{itemize}
	\end{enumerate}
\item\label{item:mainTori} If $n<2$, terminate and return \False;
otherwise, terminate and return \Unknown.
\end{enumerate}
\end{algorithm}

\begin{theorem}\label{thm:seifertAlgm}
If Algorithm~\ref{algm:seifert} returns \False,
then the input edge $e$ cannot be a Seifert fibre edge.
\end{theorem}

\begin{proof}
We first note that if $e$ is isotopic to an exceptional fibre,
then the triangulation $\mathcal{D}$ will be fibred over a disc with two exceptional fibres.
In this case, we can use the same argument as in the proof of Proposition~\ref{prop:cutTwoSolidTori}
to show that Algorithm~\ref{algm:seifert} will return \Unknown\ in step~\ref{item:cutAnnulus}.
(Note that in this case, $\mathcal{D}$ has no vertical essential tori,
which is why we need to wait until step~\ref{item:notEnoughTori}
to check whether the list $\mathcal{B}$ contains at least two tori.)

Thus, for the rest of this proof, we assume that $e$ is isotopic to a regular fibre,
in which case $\mathcal{D}$ is fibred over a disc with three exceptional fibres.
This implies that $\mathcal{D}$ is $\partial$-irreducible, and hence that the algorithm will not return \False\ in step~\ref{item:mainEnumerate}.

To see that the algorithm will not return \False\ in steps~\ref{item:emptyAnnulusList}
and~\ref{item:cutTwoSolidTori}, we recall part~\ref{regFibre:annuli} of Observation~\ref{obs:regFibre}:
$\mathcal{D}$ has exactly three vertical essential annuli, each of which cuts $\mathcal{D}$ into
a solid torus piece $\mathcal{P}_0$ and a piece $\mathcal{P}_1$ that is fibred over a disc with two exceptional fibres.
By Proposition~\ref{prop:essAnnTor}, at least one of these vertical annuli
appears as either a vertex normal surface or the double of a vertex normal surface.
Thus, at some point in step~\ref{item:mainEnumerate}, we will add such an annulus $S$ to the list $\mathcal{A}$,
which implies that the algorithm will not return \False\ in step~\ref{item:emptyAnnulusList}.
Moreover, at some point in step~\ref{item:cutAnnulus}, we will cut along $S$,
and hence decompose $\mathcal{D}$ into the pieces $\mathcal{P}_0$ and $\mathcal{P}_1$ mentioned earlier.
The piece $\mathcal{P}_1$ will be added to the list $\mathcal{L}$,
which means that at some point in step~\ref{item:cutTwoSolidTori},
we will run Algorithm~\ref{algm:cutTwoSolidTori} on input $\mathcal{P}_1$;
by Proposition~\ref{prop:cutTwoSolidTori}, Algorithm~\ref{algm:cutTwoSolidTori}
cannot output \False\ in this case, so our main algorithm also cannot return \False\ at this point.

All that remains is to show that the algorithm will not return \False\ in
steps~\ref{item:notEnoughTori} and~\ref{item:mainTori}.
For this, we use part~\ref{regFibre:tori} of Observation~\ref{obs:regFibre}:
$\mathcal{D}$ has exactly three vertical essential tori, each of which cuts $\mathcal{D}$ into
a piece $\mathcal{Q}_0$ that is fibred over an annulus with one exceptional fibre and
a piece $\mathcal{Q}_1$ that is fibred over a disc with two exceptional fibres.
By Proposition~\ref{prop:essAnnTor}, at least two of these vertical tori appear as
either vertex normal surfaces or doubles of vertex normal surfaces.
Thus, at some point in step~\ref{item:mainEnumerate}, we will add two such tori $S_0$ and $S_1$ to the list $\mathcal{B}$,
which implies that the algorithm will not return \False\ in step~\ref{item:notEnoughTori}.
Moreover, for each $i\in\{0,1\}$, we will cut along $S_i$ at some point in step~\ref{item:cutTori},
and hence decompose $\mathcal{D}$ into the pieces $\mathcal{Q}_0$ and $\mathcal{Q}_1$ mentioned earlier.
We subsequently run Algorithm~\ref{algm:cutOneSolidTorus} on $\mathcal{Q}_0$ and
Algorithm~\ref{algm:cutTwoSolidTori} on $\mathcal{Q}_1$;
by Propositions~\ref{prop:cutOneSolidTorus} and~\ref{prop:cutTwoSolidTori}, respectively,
neither of these subroutines will output \False.
Thus, we increment the variable $n$ once for each $i\in\{0,1\}$,
which implies that our main algorithm cannot return \False\ in step~\ref{item:mainTori}.
\end{proof}

Implementations of
Algorithms~\ref{algm:cutTwoSolidTori},~\ref{algm:cutOneSolidTorus} and~\ref{algm:seifert}
are available at \url{https://github.com/AlexHe98/triang-counterex}.
The only aspect of these implementations that we have not yet mentioned is that
we include an alternative technique for testing whether a $3$-manifold is Seifert fibred in a particular way,
as this sometimes allows us to avoid the computationally expensive task of enumerating vertex normal surfaces.
Specifically, between steps~\ref{item:drill} and~\ref{item:mainLists} of Algorithm~\ref{algm:seifert},
our implementation attempts to detect a Seifert fibration for the triangulation $\mathcal{D}$ using combinatorial recognition;
that is, we perform 2-3 and 3-2 moves on $\mathcal{D}$ with the goal of obtaining a standard triangulation
of a Seifert fibre space that is recognised by \Regina's \texttt{BlockedSFS} class~\cite{Regina}.

\subsection{Tracking edges as we perform 2-3 and 3-2 moves}\label{subsec:pachner}

Suppose that in some triangulation $\mathcal{T}$, we already know
which edges are (for instance) core edges.
Recall from Section~\ref{subsec:elemMoves} that if we generate a new triangulation $\mathcal{T}'$ by
performing a 3-2 move about an edge $e$ of $\mathcal{T}$,
then the only difference between the $1$-skeletons of $\mathcal{T}$ and $\mathcal{T}'$
is that $\mathcal{T}'$ no longer contains the edge $e$.
Thus, in principle, we do not need to directly compute which edges of $\mathcal{T}'$ are core edges;
we can just recover this information from what we already know about $\mathcal{T}$.

In the other direction, if we generate a new triangulation $\mathcal{T}'$ by
performing a 2-3 move, then recall that the only difference between the $1$-skeletons
of $\mathcal{T}$ and $\mathcal{T}'$ is that $\mathcal{T}'$ contains a new edge $e$.
In this case, we need to check whether $e$ is a core edge,
but there should be no need to recompute this for the other edges.

In practice, the situation is complicated by the fact that performing
2-3 and 3-2 moves in \Regina\ could arbitrarily renumber the edges.
Our solution is to use a bespoke implementation of 2-3 and 3-2 moves that provides,
as part of the output, a description of how the edges are renumbered.
The source code is available at \url{https://github.com/AlexHe98/triang-counterex}.
This implementation---and more generally, this idea of tracking how
an elementary move renumbers the vertices, edges or faces of a triangulation---may have other applications.

\section{Removing bad edges using a targeted search}\label{sec:heuristics}

This section discusses the algorithm that we used to search for counterexamples to
Conjectures~\ref{conj:core},~\ref{conj:tunnel} and~\ref{conj:fibre}.
To describe this search algorithm, we need to treat core edges, tunnel edges and Seifert fibre edges in a unified way.
For this, we note that each of these types of edges is defined by a property that is invariant under ambient isotopy;
let $P$ denote any such property.
We will say that an edge $e$ in a one-vertex triangulation is \textbf{bad} if $e$ satisfies this property $P$, and \textbf{good} otherwise;
throughout the rest of this paper, whenever we speak of good and bad edges,
the specific choice of the property $P$ will either be unimportant or clear from context.

With this terminology, we can rephrase the task of finding one of the desired counterexamples as the task of finding a triangulation with no bad edges.
Our strategy for achieving this stems from the following observation, which we mentioned in Section~\ref{subsec:elemMoves}:
when we perform a 3-2 move about an edge $e$, the only effect that this has on the $1$-skeleton is to remove $e$.
This suggests a natural way to turn a triangulation with one or more bad edges into a triangulation with no bad edges:
find a sequence of 2-3 and 3-2 moves that allows us to remove each of the bad edges using a 3-2 move.
In other words, the idea is to search the Pachner graph for a triangulation with no bad edges.

As mentioned in Section~\ref{subsec:counterex}, it is crucial to perform this search in a targeted fashion (rather than performing a brute-force search).
Sections~\ref{subsec:defects} and~\ref{subsec:targEnum} give a high-level overview of our targeted search algorithm.
We then discuss various implementation details in
Sections~\ref{subsec:multi} and~\ref{subsec:otherImp}.

\subsection{Heuristics for removing bad edges}\label{subsec:defects}

To perform a targeted search for a triangulation with no bad edges, we require heuristics for measuring how far away we are from such a triangulation.
Our goal now is to motivate and define the heuristics that we use.

The most obvious quantity to measure is the number of bad edges;
certainly, it would seem reasonable for our search to avoid increasing this quantity.
Since 3-2 moves only change the $1$-skeleton by removing an edge, we see that such a move can never increase the number of bad edges.
However, 2-3 moves change the $1$-skeleton by \emph{creating} a new edge,
so if the new edge happens to be bad then we will have increased the number of bad edges;
for this reason, one of the critical subroutines in our search algorithm is testing whether the
edge introduced by a 2-3 move is bad, so that we know which 2-3 moves we should try to avoid.

This alone is not particularly helpful, since there are many triangulations with the same number of bad edges.
To motivate some more fine-grained heuristics, we reiterate that the only way we can reduce the number of bad edges
is to perform a 3-2 move about a bad edge.
Such a move is possible if and only if there exists a bad edge of degree three that actually meets three distinct tetrahedra.
This suggests that we should measure the following quantities for each bad edge:

\begin{definitions}
Let $e$ be an edge in a one-vertex triangulation.
Let $d(e)$ denote the degree of $e$, and let $n(e)$ denote the number of \emph{distinct} tetrahedra that meet $e$.
\begin{itemize}[nosep]
\item The \textbf{degree defect} of $e$, denoted $\delta(e)$, is given by $\lvert d(e)-3 \rvert$.
\item The \textbf{multiplicity defect} of $e$, denoted $\mu(e)$, is given by $d(e)-n(e)$.
\end{itemize}
\end{definitions}

Note that the degree defect and multiplicity defect are both always non-negative,
and that we can perform a 3-2 move about an edge $e$ if and only if
both of these defects are equal to zero.
Thus, to reduce the number of bad edges, it seems reasonable to greedily prioritise triangulations in which the bad edges have small degree defect and multiplicity defect.
This motivates the following definition:

\begin{definition}[Complexity of a triangulation]\label{def:complexity}
Let $\mathcal{T}$ be a one-vertex triangulation with at least one bad edge, and let $e_1,\ldots,e_k$ denote the bad edges in $\mathcal{T}$.
The \textbf{complexity} of $\mathcal{T}$ is given by the $4$-tuple
\[
\left(\;
k,\;
\max_{1\leqslant i\leqslant k}\mu(e_i),\;
\max_{1\leqslant i\leqslant k}\delta(e_i),\;
\lvert\mathcal{T}\rvert
\;\right).
\]
\end{definition}

Our search algorithm attempts to minimise the complexity of a triangulation
with respect to the \emph{lexicographical ordering}.
This is why the number $k$ of bad edges appears as the first entry:
our primary objective is to minimise this number.
Our secondary objective, as discussed above, is to reduce the
multiplicity defect and degree defect of the bad edges,
in the hope that this will allow us to remove these edges using 3-2 moves.
We take the maximum defects across all bad edges (as opposed to, say, the minimum)
because our ultimate goal is to remove all of these edges, not just one of them.

It is crucial that we prioritise reducing the multiplicity defect ahead of reducing the degree defect.
The reason for this is best illustrated by recounting our initial approach to this problem.
At first, our complexity did not involve the multiplicity defect at all.
With this na\"{i}ve notion of complexity, we found that the search could easily reach
a region of the Pachner graph where the bad edges had degree defect close to $0$;
however, the search would then get stuck enumerating lots of triangulations without ever
succeeding in reducing the degree defect to $0$.
We eventually realised that the search was getting trapped in a region of the Pachner graph where
the bad edges had multiplicity defect equal to $2$.
Placing a high priority on reducing the multiplicity defect
gives the search some impetus to avoid such regions.

Finally, all else being equal, it would be nice for our counterexamples to be as small as possible.
This is why we include the size of the triangulation as the final entry of the complexity.

\subsection{The targeted search algorithm}\label{subsec:targEnum}

\begin{algorithm}\label{algm:targEnum}
This algorithm takes the following inputs:
\begin{itemize}[nosep]
\item A one-vertex triangulation $\mathcal{T}$ with $n$ bad edges, where $n\geqslant1$.
\item A non-negative integer $x$
(the number of extra bad edges that we are allowed to use).
\end{itemize}
To search for a new triangulation with $n-1$ bad edges:
\begin{enumerate}[nosep, label={(\arabic*)}]
\item For each bad edge in $\mathcal{T}$, check whether it is possible to perform a 3-2 move about this edge.
If such a move is possible, terminate and return the triangulation that results from performing this 3-2 move.
\item Create a set $\mathcal{S}$.
Also create a priority queue $\mathcal{Q}$ that stores triangulations in order of increasing complexity.
Add $\mathcal{T}$ to both $\mathcal{S}$ and $\mathcal{Q}$.
\item\label{item:parallel} While $\mathcal{Q}$ is non-empty:
	\begin{enumerate}[nosep, label={(\roman*)}]
	\item Remove the first triangulation $\mathcal{F}$
	(i.e., a triangulation with smallest complexity) from $\mathcal{Q}$,
	and let $m$ denote the number of bad edges in $\mathcal{F}$.
	\item For a 3-2 move about an edge $e$, call this move \textbf{eligible} if $e$ is a good edge;
	for a 2-3 move, let $e$ denote the new edge that is created by this move,
	and call this move \textbf{eligible} if either
	$e$ is a good edge, or $e$ is a bad edge but $m<n+x$.
	For each eligible move on $\mathcal{F}$, check whether (up to combinatorial isomorphism)
	the set $\mathcal{S}$ already contains the triangulation $\mathcal{G}$ that
	we obtain after performing this move. If not:
		\begin{itemize}[nosep]
		\item Add $\mathcal{G}$ to both $\mathcal{S}$ and $\mathcal{Q}$.
		\item Perform all possible sequences of 3-2 moves about bad edges in $\mathcal{G}$.
		For each triangulation $\mathcal{T}'$ that we obtain from such moves,
		if $\mathcal{T}'$ does not already appear in the set $\mathcal{S}$ (up to combinatorial isomorphism),
		then add $\mathcal{T}'$ to both $\mathcal{S}$ and $\mathcal{Q}$.
		\item If we can find such a sequence consisting of $(m-n+1)$ 3-2 moves,
		then the final triangulation $\mathcal{T}^{\ast}$ in this sequence has $n-1$ bad edges.
		Terminate and return $\mathcal{T}^{\ast}$.
		\end{itemize}
	\end{enumerate}
\end{enumerate}
\end{algorithm}

Visit \url{https://github.com/AlexHe98/triang-counterex} to see our implementation of Algorithm~\ref{algm:targEnum}.
We discuss the major details of this implementation in Sections~\ref{subsec:multi} and~\ref{subsec:otherImp}.

It is worthwhile to explain why we terminate the search as soon as we find a triangulation with $n-1$ bad edges,
rather than allowing the search to continue until it (hopefully) finds a triangulation with no bad edges at all.
This is partly a remnant of our early investigations into the conjectures from Section~\ref{subsec:conj};
initially, it was not at all clear whether (for instance) triangulations with no core edges even existed,
so we were happy to make any kind of progress towards reducing the number of core edges.
One reason we have persisted with terminating the search early is that we have, on occasion,
found it useful to have the flexibility to make adjustments as we progressively lower the number of bad edges.
Examples of such adjustments include:
\begin{itemize}
\item changing the number $x$ of extra bad edges that we allow Algorithm~\ref{algm:targEnum} to use
(though we have not encountered an example where this particular adjustment was helpful);
\item changing the number of concurrent processes that we use in step~\ref{item:parallel} of Algorithm~\ref{algm:targEnum}
(we discuss this in more detail in Section~\ref{subsec:multi},
and give examples in Section~\ref{subsec:remCore} and Appendix~\ref{appen:remTunMore}); or
\item even changing how we measure the complexity of a triangulation
(see Section~\ref{subsec:remFibre} for an example where this was useful).
\end{itemize}

\subsection{Troublesome regions, concurrent computation, and instability}\label{subsec:multi}

In Section~\ref{subsec:defects}, we mentioned regions of the Pachner graph where
the bad edges have very low degree defect, but multiplicity defect equal to $2$
(such edges naturally occur, for instance, at the hearts of layered solid tori).
Although placing a high priority on reducing multiplicity defect does help the search
avoid such regions, if the search nevertheless gets trapped in such a region then it can be quite difficult to escape:
the only way out is to perform moves that \emph{increase} the degree defect of the bad edges,
which goes completely against the grain of what our targeted search attempts to prioritise.

This phenomenon occurs more generally: a locally optimal move
can send the search into a ``troublesome'' region (which could, a priori, be \emph{infinite}) of the Pachner graph where
none of the subsequently available moves decrease the complexity of the triangulation.
In the case discussed above, this causes the search to get stuck enumerating lots of triangulations
with similar complexity;
as we will see in Section~\ref{subsec:remCore}, there are also cases where
the search can start enumerating lots of triangulations with \emph{rapidly increasing} complexity.

Algorithm~\ref{algm:targEnum} is especially vulnerable to falling into
such troublesome regions if we deal with triangulations one at a time in step~\ref{item:parallel}.
However, when we instead use multiple processes to deal with several triangulations concurrently,
the search is sometimes able to either avoid or escape these troublesome regions.
There are probably two drivers for this:
\begin{enumerate}[label={(\arabic*)}]
\item using multiple processes causes the search to explore with more ``breadth'' than a purely greedy approach; and
\item the search gains some randomness because the order
in which triangulations are inserted into the priority queue $\mathcal{Q}$ could vary each time we run the search.
\end{enumerate}
There may be more direct methods to achieve similar behaviour;
our method was good enough, and had a low cost
(both in human effort and in computational complexity).

One drawback is that it is impossible to know in advance how many concurrent processes we should use.
In fact, as we will see in Section~\ref{sec:results}, there is often a ``sweet spot'' where
running the search will produce a triangulation $\mathcal{T}^\ast$ with the desired number of bad edges,
but if we significantly increase or decrease the number of processes,
then the search tends to get trapped in a troublesome region.

Our solution is to implement Algorithm~\ref{algm:targEnum} so that it
periodically prints an update on the complexity of the triangulations that
it is currently dealing with, and to manually check that the search is
making satisfactory progress towards reducing the complexity.
If we do not observe such progress, then we simply restart the search,
possibly with a change to the number of concurrent processes that we use.
Although it is not ideal that some human intervention is required,
this is not unheard of for difficult problems in computational topology that
only need to be successfully solved once (such as finding a counterexample).
Another example of this is the work done in~\cite{Burton2020}
to extend the census of prime knots up to $19$ crossings.

The last thing that we mention here is that our implementation is somewhat unstable.
Indeed, we have already noted that the search can produce different results when we
use a different number of concurrent processes, or when we simply rerun the search.
We have encountered one other manifestation of this instability:
even refactoring the source code can significantly change the results of the search.
We have not managed to find a way to modify the algorithm to make it more stable.

\subsection{Other implementation details}\label{subsec:otherImp}

We now outline some other important details of our implementation of Algorithm~\ref{algm:targEnum}.

First, as mentioned in Section~\ref{subsec:pachner}, we use a bespoke implementation of 2-3 and 3-2 moves
which allows us to minimise how often we need to check whether an edge is bad.
This is crucial because for our purposes, checking whether an edge is bad
generally relies on computationally expensive techniques involving normal surfaces:
\begin{itemize}
\item To recognise a core edge or a tunnel edge, we need to delete a small regular neighbourhood of the edge,
and test whether the resulting $3$-manifold is a solid torus or a genus-$2$ handlebody, respectively.
For this, we use handlebody recognition (Algorithm~\ref{algm:handlebody}).
\item To recognise a Seifert fibre edge, we rely on Algorithm~\ref{algm:seifert}.
\end{itemize}

Second, since Algorithm~\ref{algm:targEnum} always considers triangulations up to
combinatorial isomorphism, we actually store isomorphism signatures in
the set $\mathcal{S}$ and the priority queue $\mathcal{Q}$.
As described in~\cite{Burton2011arXiv}, isomorphism signatures are
a standard and indispensable tool when exploring the Pachner graph.

Third, when the priority queue $\mathcal{Q}$ contains multiple triangulations with the same complexity,
we choose to prioritise such triangulations by insertion order.

Finally, in practice, we usually set the input variable $x$ to be $0$,
in which case Algorithm~\ref{algm:targEnum} discards all triangulations with
more bad edges than the initial input triangulation.
Having said this, there are occasions where taking $x>0$ (and hence allowing
the number of bad edges to increase) may be beneficial;
we discuss this in a little more detail in Section~\ref{subsec:remFibre}.

\section{The counterexamples}\label{sec:results}

In Sections~\ref{subsec:remCore},~\ref{subsec:remTunnel} and~\ref{subsec:remFibre},
we discuss experimental details specific to Conjectures~\ref{conj:core},~\ref{conj:tunnel} and~\ref{conj:fibre}, respectively;
in particular, as mentioned in Section~\ref{sec:intro}, we give counterexamples to all three of these conjectures.
Moreover, we show in Section~\ref{subsec:infinite} that each of our counterexamples
can be turned into an infinite family of counterexamples.

All of the computations that we discuss in this section were run on a laptop with
an Intel Core i5-7200U processor, which has just two physical cores divided into four logical processors.
It is therefore remarkable that we were able to obtain almost all of our counterexamples in
no more than a few \emph{minutes} of wall time; even the worst example that we present
(see Figure~\ref{fig:remTun7_7} in Appendix~\ref{appen:remTunMore})
required less than $34$ minutes of wall time.
This is mostly due to the fact that our targeted search was able to home in on an extremely small
portion of the search space:
for each $3$-manifold for which we successfully found a counterexample,
we never needed to enumerate more than a few thousand triangulations.

\subsection{Triangulations with no core edges}\label{subsec:remCore}

We first discuss the special case of the $3$-sphere;
recall that in this case, core edges are equivalent to \emph{unknotted edges}.
We saw in Section~\ref{subsec:counterex} that the $3$-sphere has \num{422533279}
one-vertex triangulations with $10$ or fewer tetrahedra,
and that all of these triangulations have a core edge (in fact, at least two core edges).
In spite of this seemingly compelling evidence, the triangulation $\mathcal{T}$ with isomorphism signature
\[
\texttt{uLLvQQvLAPvPAQccdfeghhgmklnorsqssttthsaaggggaaaaaaanaaagb}
\]
turns out to be a $20$-tetrahedron one-vertex $3$-sphere with no core edges.

To find this counterexample, we began by arbitrarily selecting
the triangulation $\mathcal{T}_2$ with isomorphism signature \texttt{cMcabbgqs};
this is a one-vertex $3$-sphere with complexity $(2,5,4,2)$ (recall Definition~\ref{def:complexity}).
After running Algorithm~\ref{algm:targEnum} twice, we obtained a $22$-tetrahedron
one-vertex triangulation $\mathcal{T}_0$ of the $3$-sphere with no core edges;
the results are summarised in Figure~\ref{fig:remCorS3}.

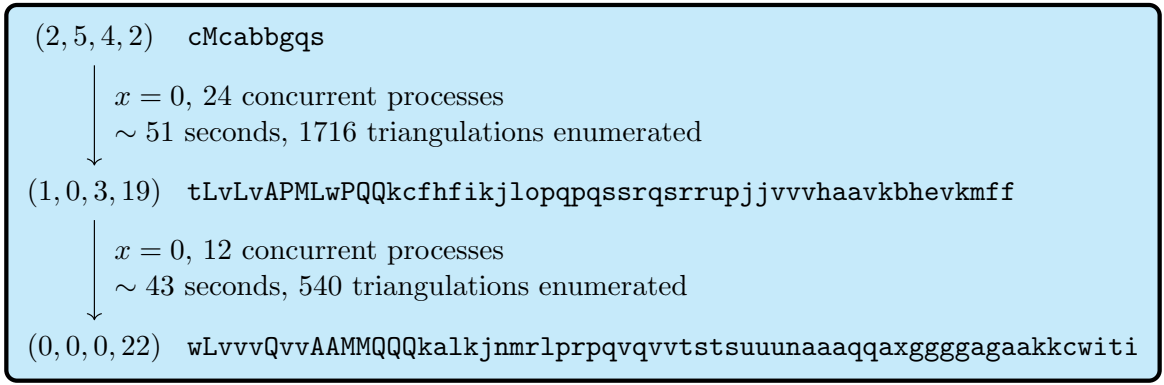
\begin{figure}[htbp]
\centering
\begin{tikzcd}[
	column sep=0ex,
	row sep=huge,
	/tikz/column 2/.append style={anchor=base west},
	every label/.append style={font=\normalsize},
	resultsbox,]
(2,5,4,2)
\ar{d}{\begin{array}{l}
\text{$x=0$, $24$ concurrent processes}\\
\text{$\sim51$ seconds, $1716$ triangulations enumerated}
\end{array}}
&
\texttt{cMcabbgqs}
\\
(1,0,3,19)
\ar{d}{\begin{array}{l}
\text{$x=0$, $12$ concurrent processes}\\
\text{$\sim43$ seconds, $540$ triangulations enumerated}
\end{array}}
&
\texttt{tLvLvAPMLwPQQkcfhfikjlopqpqssrqsrrupjjvvvhaavkbhevkmff}
\\
(0,0,0,22)
&
\texttt{wLvvvQvvAAMMQQQkalkjnmrlprpqvqvvtstsuuunaaaqqaxggggagaakkcwiti}
\end{tikzcd}
\caption{Removing core edges from \texttt{cMcabbgqs} ($3$-sphere).}
\label{fig:remCorS3}
\end{figure}

To turn $\mathcal{T}_0$ into the $20$-tetrahedron example $\mathcal{T}$,
we ran a breadth-first search through the Pachner graph,
but with the restriction that we ignored any 2-3 move that introduces a core edge.
This took approximately $367$ seconds of wall time.
We also know that we cannot reduce the size further (still ignoring 2-3 moves that introduce core edges)
without passing through at least one triangulation with more than $23$ tetrahedra;
checking this required an additional $228$ seconds of wall time.

It is worth noting that to remove one core edge from $\mathcal{T}_2$, we tried
running Algorithm~\ref{algm:targEnum} with different numbers of concurrent processes,
but using around $24$ processes seems to produce the best results.
In particular, we initially tried using $12$ processes, but this
produced a triangulation with isomorphism signature
\[
\texttt{sLvAAvLAzMMQQcdceflkmjmqonprqprrhvrqnkkkksqeekocksf}
\]
and complexity $(1,2,1,18)$.
The core edge in this triangulation has multiplicity defect $2$ but degree defect $1$;
this is exactly one of the troublesome cases that we mentioned in Section~\ref{subsec:multi}.
We needed to increase the number of processes to $24$ to avoid this.

Somewhat surprisingly, increasing the number of processes significantly beyond $24$
also appears to adversely affect the effectiveness of Algorithm~\ref{algm:targEnum}.
For example, with $36$ processes, although the search is sometimes able to
remove a core edge, it seems to do so less reliably.
Instead, we find that the search has a tendency to get trapped in a different type of troublesome region:
\begin{itemize}
\item After about $10$ seconds, the search reaches a triangulation with complexity $(2,4,16,15)$.
\item After about $20$ seconds, the search reaches a triangulation with complexity $(2,4,23,22)$.
\item After about $30$ seconds, the search reaches a triangulation with complexity $(2,4,26,25)$.
\end{itemize}
When this happens, the complexity only appears to increase further if we allow the search to continue.

Beyond $3$-spheres, we also found one-vertex triangulations with no core edges for
the following lens spaces: $L_{6,1}$, $L_{9,2}$, $L_{11,2}$, $L_{13,3}$, and $L_{16,3}$.
See Appendix~\ref{appen:remCoreMore} for detailed results.

\subsection{Triangulations with no tunnel edges}\label{subsec:remTunnel}

All fourteen prime knots with crossing number up to $7$ have tunnel number equal to one;
here, we specify these knots using their Rolfsen names.
Except for the $5_2$ knot, we were able to find one-vertex ideal triangulations
with no tunnel edges for all of these knots:
\begin{itemize}
\item For the trefoil knot (i.e., the $3_1$ knot), we arbitrarily selected
the isomorphism signature \texttt{cPcbbbadu}.
Figure~\ref{fig:remTunTrefoil} summarises the result of running Algorithm~\ref{algm:targEnum}.
	\begin{figure}[htbp]
	\centering
	\begin{tikzcd}[
		column sep=0ex,
		row sep=large,
		/tikz/column 2/.append style={anchor=base west},
		every label/.append style={font=\normalsize},
		resultsbox,]
	(1,8,7,2)
	\ar{d}{
	\text{$\sim58$ seconds, $307$ triangulations enumerated}}
	&
	\texttt{cMcbbbadu}
	\\
	(0,0,0,11)
	&
	\texttt{lLLLzzQQcbcgeijgjikkktsltaurattgg}
	\end{tikzcd}
	\caption{Removing tunnel edges from \texttt{cPcbbbadu} (trefoil knot); $x=0$, $8$ processes.}
	\label{fig:remTunTrefoil}
	\end{figure}
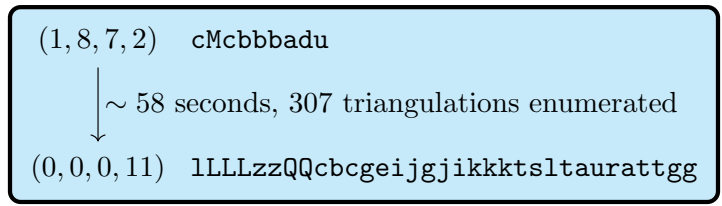
\item For the figure-eight knot (i.e., the $4_1$ knot), we arbitrarily selected
the isomorphism signature \texttt{cPcbbbiht}.
Figure~\ref{fig:remTunFigure8} summarises the results of running Algorithm~\ref{algm:targEnum} twice.
	\begin{figure}[htbp]
	\centering
	\begin{tikzcd}[
		column sep=0ex,
		row sep=large,
		/tikz/column 2/.append style={anchor=base west},
		every label/.append style={font=\normalsize},
		resultsbox,]
	(2,4,3,2)
	\ar{d}{
	\text{$\sim6$ seconds, $26$ triangulations enumerated}}
	&
	\texttt{cPcbbbiht}
	\\
	(1,6,7,5)
	\ar{d}{
	\text{$\sim224$ seconds, $903$ triangulations enumerated}}
	&
	\texttt{fLLQcacdedejbqqww}
	\\
	(0,0,0,12)
	&
	\texttt{mLvzALAQQccefhijliklklhnipouapufbvv}
	\end{tikzcd}
	\caption{Removing tunnel edges from \texttt{cPcbbbiht} (figure-eight knot); $x=0$, $24$ processes.}
	\label{fig:remTunFigure8}
	\end{figure}
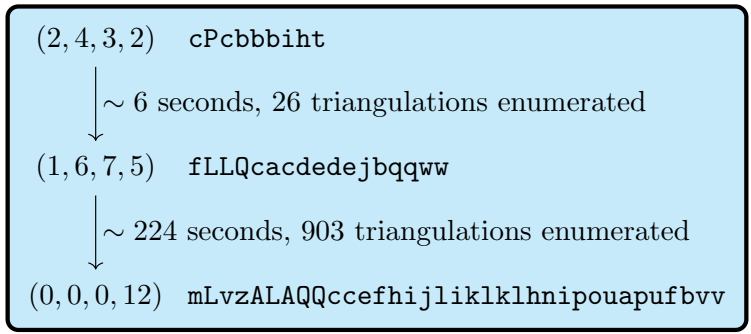
\item For the $(5,2)$ torus knot (i.e., the $5_1$ knot), we arbitrarily selected
the isomorphism signature \texttt{dLQbcccaekv}.
Figure~\ref{fig:remTunTorus52} summarises the results of running Algorithm~\ref{algm:targEnum} twice.
	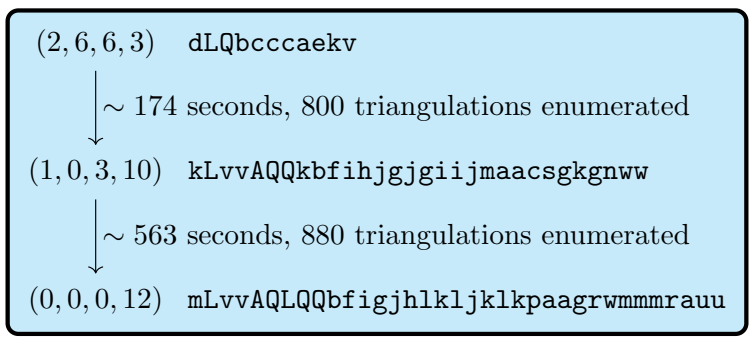
\begin{figure}[htbp]
	\centering
	\begin{tikzcd}[
		column sep=0ex,
		row sep=large,
		/tikz/column 2/.append style={anchor=base west},
		every label/.append style={font=\normalsize},
		resultsbox,]
	(2,6,6,3)
	\ar{d}{
	\text{$\sim174$ seconds, $800$ triangulations enumerated}}
	&
	\texttt{dLQbcccaekv}
	\\
	(1,0,3,10)
	\ar{d}{
	\text{$\sim563$ seconds, $880$ triangulations enumerated}}
	&
	\texttt{kLvvAQQkbfihjgjgiijmaacsgkgnww}
	\\
	(0,0,0,12)
	&
	\texttt{mLvvAQLQQbfigjhlkljklkpaagrwmmmrauu}
	\end{tikzcd}
	\caption{Removing tunnel edges from \texttt{dLQbcccaekv} ($(5,2)$ torus knot); $x=0$, $24$ processes.}
	\label{fig:remTunTorus52}
	\end{figure}
\item We also found one-vertex ideal triangulations with no tunnel edges for all ten prime knots with crossing number equal to $6$ or $7$.
See Appendix~\ref{appen:remTunMore} for detailed results.
\end{itemize}

One striking observation is that although the triangulations here are relatively small,
and although we only enumerated relatively few triangulations,
our running times are nevertheless significantly longer than those in Section~\ref{subsec:remCore}.
The main reason for this is probably that recognising genus-$2$ handlebodies is
more expensive than recognising solid tori.

\subsection{Triangulations with no Seifert fibre edges}\label{subsec:remFibre}

Let $\mathcal{M}$ denote the small Seifert fibre space
$\mathcal{S}\left( \frac{1}{2}, \frac{2}{3}, -\frac{1}{3} \right)$
(recall Construction~\ref{cons:sfs}).
We found the following $11$-tetrahedron one-vertex triangulation $\mathcal{T}$ of $\mathcal{M}$,
which turns out to have no Seifert fibre edges:
\[
\texttt{lLLLLPMQccddfjiihikkkpkrwaaacttvc}.
\]
To do this, we began with the isomorphism signature \texttt{fLLQcaceeedjkuxkj}.
Running Algorithm~\ref{algm:targEnum} four times with $x=1$ and $12$ concurrent processes produced
a $13$-tetrahedron triangulation $\mathcal{T}^\ast$ with no Seifert fibre edges;
the results are summarised in Figure~\ref{fig:remFib}.
We then used a breadth-first search (similar to the one from Section~\ref{subsec:remCore})
to turn $\mathcal{T}^\ast$ into the $11$-tetrahedron triangulation $\mathcal{T}$.

\begin{figure}[htbp]
\centering
\begin{tikzcd}[
	column sep=0ex,
	row sep=large,
	/tikz/column 2/.append style={anchor=base west},
	every label/.append style={font=\normalsize},
	resultsbox,]
(4,2,2,5)
\ar{d}{
\text{$\sim3$ seconds, $25$ triangulations enumerated}}
&
\texttt{fLLQcaceeedjkuxkj}
\\
(3,2,3,6)
\ar{d}{
\text{$\sim6$ seconds, $344$ triangulations enumerated}}
&
\texttt{gLLAQbdedfffendolgn}
\\
(2,2,4,8)
\ar{d}{
\text{$\sim12$ seconds, $664$ triangulations enumerated}}
&
\texttt{ivLLQQccehgfgfhhjsquaaagj}
\\
(1,0,2,12)
\ar{d}{
\text{$\sim15$ seconds, $251$ triangulations enumerated}}
&
\texttt{mLLwPvMQQacdhghklkjlklnkamamvirvlji}
\\
(0,0,0,13)
&
\texttt{nLvPwLzQQkccgfiikjmklmlmhnahlupmtrsvgb}
\end{tikzcd}
\caption{Removing Seifert fibre edges from \texttt{fLLQcaceeedjkuxkj}
($\mathcal{M} = \mathcal{S}\left( \frac{1}{2}, \frac{2}{3}, -\frac{1}{3} \right)$);
$x=1$, $12$ processes.}
\label{fig:remFib}
\end{figure}
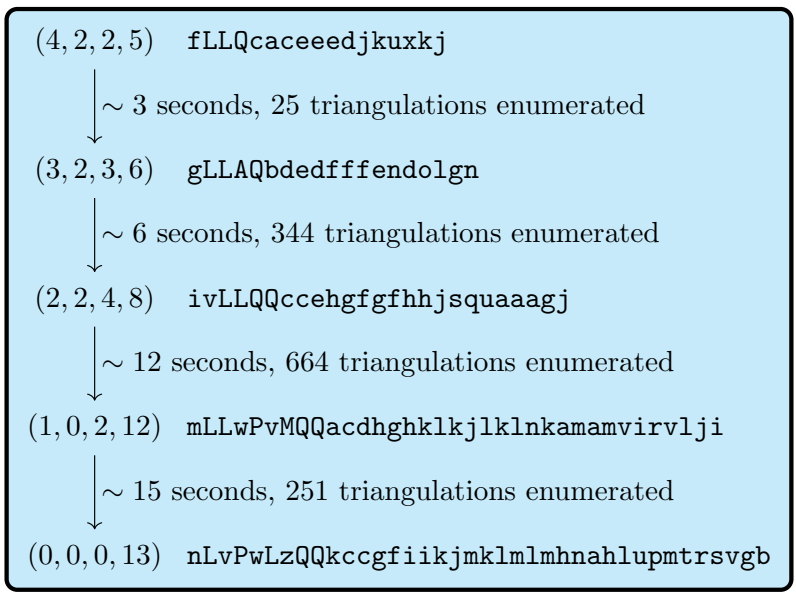

The reason for taking $x=1$ instead of $x=0$ comes from older versions of the code.
Early versions of Algorithm~\ref{algm:seifert} included fewer normal surfaces in the analysis,
making it less effective at identifying good edges (in this context, edges not isotopic to Seifert fibres).
At the time, we encountered triangulations that were ``dead ends'', in the following sense:
every 2-3 move would create an edge that we could not certify as good.
Our short-term solution was to allow the search to introduce extra bad edges if necessary.
Having since strengthened Algorithm~\ref{algm:seifert}, taking $x=1$ no longer seems necessary
for the example that we just discussed; we did not revert to using $x=0$ simply because we found no compelling reason to do so.

It is also worth mentioning that we encountered one example where even taking $x=1$ was not enough to prevent ``dead ends''
(even with the current implementation of Algorithm~\ref{algm:seifert}).
Specifically, this happens for the triangulation with isomorphism signature \texttt{fvPQcdecedekrsnrs}
and complexity $(6,0,2,5)$; this represents the \emph{Poincar\'{e} homology sphere}, which admits the Seifert fibration
$\mathcal{S}\left( \frac{1}{2}, \frac{1}{3}, -\frac{4}{5} \right)$.
When we increased the value of $x$ to $2$, we were able to reach a triangulation with four Seifert fibre edges,
but we did not make any further progress.

On a related note, we did manage to find a counterexample for one other small Seifert fibre space,
namely $\mathcal{S}\left( \frac{1}{2}, \frac{1}{3}, \frac{2}{3} \right)$,
which we denote by $\mathcal{M}'$.
Specifically, we found the following $20$-tetrahedron triangulation $\mathcal{T}'$ of $\mathcal{M}'$,
which has no Seifert fibre edges:
\[
\texttt{uLLPvvLQLAPAPQccdfeknlipnomqnrtstssthsipaaliaravlkuaaxxxx}.
\]
To do this, we began with the isomorphism signature \texttt{gLLMQacdefefjkaknkr},
which has complexity $(4,2,4,6)$.
However, we were not able to remove all four Seifert fibre edges by simply running Algorithm~\ref{algm:targEnum} four times.
Instead, to remove the first and third such edges, we needed to modify how our search measured the complexity of our triangulations.

To describe this modification, consider a one-vertex triangulation $\mathcal{S}$ of $\mathcal{M}'$,
and let $e_1,\ldots,e_k$ denote all the bad edges in $\mathcal{S}$
(where, ``bad'' in this context means that the edge is a Seifert fibre edge).
We should imagine that $\mathcal{S}$ is some intermediate triangulation that we encounter in a search
similar to Algorithm~\ref{algm:targEnum}, so in the background there is some initial input triangulation that has $n$ bad edges,
and there is also some non-negative integer $x$ such that the search discards any triangulations with more than $n+x$ bad edges.
However, unlike in Algorithm~\ref{algm:targEnum}, we temporarily forget that
our ultimate goal is to remove all $k$ bad edges in $\mathcal{S}$, and instead focus entirely on
the intermediate goal of removing $k-n+1$ bad edges from $\mathcal{S}$
(and hence reducing the number of bad edges to $n-1$).
To this end, we single out two (not necessarily distinct) bad edges as follows:
\begin{itemize}
\item First, arrange the bad edges in order of increasing \emph{multiplicity defect},
and take $e_m$ to be the $(k-n+1)$st bad edge in this ordering.
\item Then, arrange the bad edges in order of increasing \emph{degree defect},
and take $e_d$ to be the $(k-n+1)$st bad edge in this ordering.
\end{itemize}
We modify Algorithm~\ref{algm:targEnum} by measuring the complexity of $\mathcal{S}$ as follows:
\[
\left(\;
k,\;
\mu(e_m),\;
\delta(e_d),\;
\lvert\mathcal{S}\rvert
\;\right).
\]

Figure~\ref{fig:remFibInt} summarises the results of removing all four Seifert fibre edges from
the triangulation of $\mathcal{M}'$ given by the isomorphism signature \texttt{gLLMQacdefefjkaknkr};
we used $x=1$ and $12$ concurrent processes at each stage.
In the stages where we used the modified complexity instead of the (unmodified) complexity introduced in
Definition~\ref{def:complexity}, Figure~\ref{fig:remFibInt} lists the unmodified complexity above the modified complexity.
We obtained a $25$-tetrahedron triangulation of $\mathcal{M}'$ with no Seifert fibre edges,
which we then simplified to the $20$-tetrahedron triangulation $\mathcal{T}'$
using the same breadth-first search technique as before.

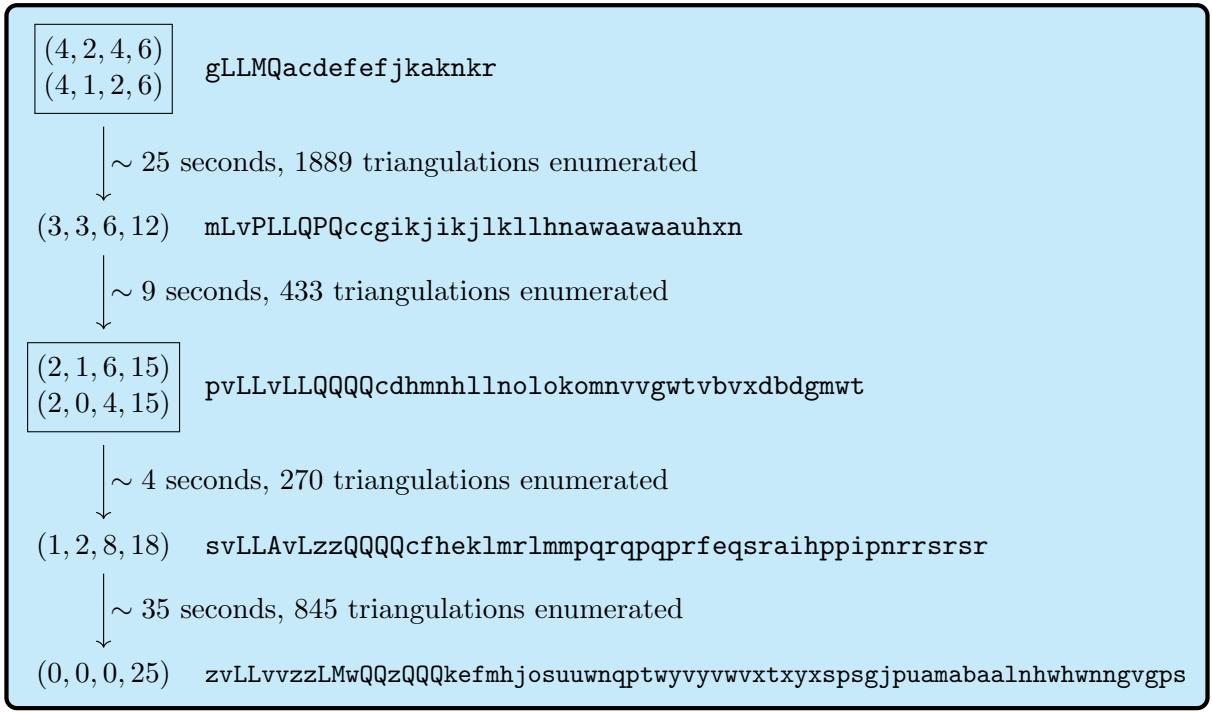
\begin{figure}[htbp]
\centering
\begin{tikzcd}[
	column sep=0ex,
	row sep=large,
	/tikz/column 2/.append style={anchor=base west},
	every label/.append style={font=\normalsize},
	resultsbox,]
\fbox{$
\begin{matrix}
(4,2,4,6)\\
(4,1,2,6)
\end{matrix}
$}
\ar{d}{
\text{$\sim25$ seconds, $1889$ triangulations enumerated}}
&
\texttt{gLLMQacdefefjkaknkr}
\\
(3,3,6,12)
\ar{d}{
\text{$\sim9$ seconds, $433$ triangulations enumerated}}
&
\texttt{mLvPLLQPQccgikjikjlkllhnawaawaauhxn}
\\
\fbox{$
\begin{matrix}
(2,1,6,15)\\
(2,0,4,15)
\end{matrix}
$}
\ar{d}{
\text{$\sim4$ seconds, $270$ triangulations enumerated}}
&
\texttt{pvLLvLLQQQQcdhmnhllnolokomnvvgwtvbvxdbdgmwt}
\\
(1,2,8,18)
\ar{d}{
\text{$\sim35$ seconds, $845$ triangulations enumerated}}
&
\texttt{svLLAvLzzQQQQcfheklmrlmmpqrqpqprfeqsraihppipnrrsrsr}
\\
(0,0,0,25)
&
\text{\small\texttt{zvLLvvzzLMwQQzQQQkefmhjosuuwnqptwyvyvwvxtxyxspsgjpuamabaalnhwhwnngvgps}}
\end{tikzcd}
\caption{Removing Seifert fibre edges from \texttt{gLLMQacdefefjkaknkr}
($\mathcal{M}' = \mathcal{S}\left( \frac{1}{2}, \frac{1}{3}, \frac{2}{3} \right)$);
$x=1$, $12$ processes.
The boxes indicate where we used the modified complexity.}
\label{fig:remFibInt}
\end{figure}

\subsection{From one counterexample to infinitely many}\label{subsec:infinite}

Here, we show that each of the above counterexamples can be turned into infinite families.
The key is the following proposition:

\begin{proposition}\label{prop:infinite}
Let $\mathcal{M}$ be either a closed $3$-manifold, or a $3$-manifold with
a single boundary component of positive genus (and no other boundary components).
In the closed case, let $\mathcal{T}$ be a one-vertex triangulation of $\mathcal{M}$;
in the bounded case, let $\mathcal{T}$ be a one-vertex ideal triangulation of $\mathcal{M}$.
Let $P$ be a property of edges of $\mathcal{T}$ that is invariant under ambient isotopy.
Call $\mathcal{T}$ \textbf{interesting} if every edge in $\mathcal{T}$ satisfies $P$.
If $\mathcal{M}$ has an interesting triangulation, then $\mathcal{M}$ has
infinitely many interesting triangulations.
\end{proposition}

\begin{proof}
Let $\mathcal{T}$ be an interesting triangulation of $\mathcal{M}$.
Fix a tetrahedron $\Delta$ of $\mathcal{T}$, and let $e$ and $f$ denote a pair of opposite edges of $\Delta$.
We obtain a new triangulation $\mathcal{T}'$ of $\mathcal{M}$ by replacing $\Delta$
with a three-tetrahedron gadget, as shown in Figure~\ref{fig:infiniteCounterex}.\footnote{
This modification is a special case of an elementary move known as a \textbf{0-2 move} or \textbf{lune move}.}
In terms of $1$-skeletons, observe that all we have done is introduce two new edges $e'$ and $f'$
such that $e'$ is isotopic to $e$ and $f'$ is isotopic to $f$.
Thus, $\mathcal{T}'$ is an interesting triangulation of $\mathcal{M}$.
Repeating this procedure indefinitely gives the desired infinite family of interesting triangulations of $\mathcal{M}$.
\end{proof}

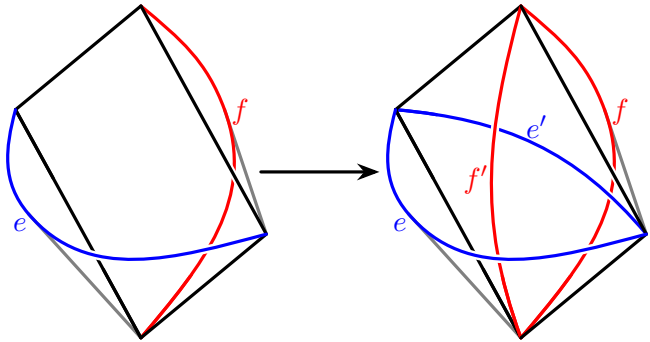
\begin{figure}[htbp]
\centering
	\begin{tikzpicture}
	\node[inner sep=0pt] (Before) at (0,0)
		{\includegraphics[scale=1]{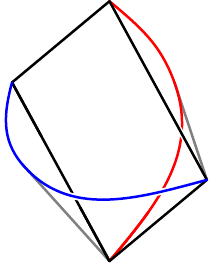}};
	\node[inner sep=0pt] (After) at (5,0)
		{\includegraphics[scale=1]{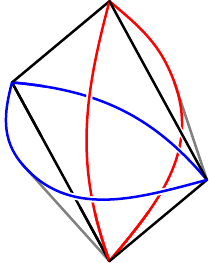}};

	\draw[very thick, line cap=round, -{Stealth}]
		($(Before.east)+(-0.1,0)$) -- ($(After.west)+(0,0)$);

	\node[blue] at ($(Before)+(-1.5,-0.7)$) {$e$};
	\node[red] at ($(Before)+(1.4,0.8)$) {$f$};

	\node[blue] at ($(After)+(-1.5,-0.7)$) {$e$};
	\node[red] at ($(After)+(1.4,0.8)$) {$f$};
	\node[blue] at ($(After)+(0.3,0.6)$) {$e'$};
	\node[red] at ($(After)+(-0.5,-0.1)$) {$f'$};
	\end{tikzpicture}
\caption{Building a new interesting triangulation.}
\label{fig:infiniteCounterex}
\end{figure}

Combining Proposition~\ref{prop:infinite} with the counterexamples from
Sections~\ref{subsec:remCore},~\ref{subsec:remTunnel} and~\ref{subsec:remFibre}
immediately yields the following three theorems:

\begin{theorem}
Let $\mathcal{M}$ be one of the following lens spaces:
the $3$-sphere, $L_{6,1}$, $L_{9,2}$, $L_{11,2}$, $L_{13,3}$, or $L_{16,3}$.
Then $\mathcal{M}$ has infinitely many one-vertex triangulations with no core edges.
\end{theorem}

\begin{theorem}
Let $K$ be a prime knot whose crossing number is at most $7$, other than the $5_2$ knot.
Then $K$ has infinitely many one-vertex ideal triangulations with no tunnel edges.
\end{theorem}

\begin{theorem}
Each of the small Seifert fibre spaces $\mathcal{S}\left( \frac{1}{2}, \frac{2}{3}, -\frac{1}{3} \right)$
and $\mathcal{S}\left( \frac{1}{2}, \frac{1}{3}, \frac{2}{3} \right)$
has infinitely many one-vertex triangulations with no Seifert fibre edges.
\end{theorem}

\section{Discussion}\label{sec:discussion}

\subsection{Unanswered questions}\label{subsec:unanswered}

Although our targeted search allowed us to produce triangulations with no core edges for six different lens spaces,
this did not work for all of the lens spaces that we tried.
For reference, we failed to find counterexamples for the following lens spaces:
$S^2\times S^1$, $\mathbb{R}P^3$, $L_{3,1}$, $L_{5,1}$, $L_{7,2}$, $L_{8,3}$,
$L_{10,3}$, $L_{11,3}$, $L_{12,5}$, $L_{13,5}$, $L_{7,1}$, $L_{14,3}$, $L_{15,4}$,
$L_{16,7}$, $L_{17,5}$, $L_{18,5}$, $L_{19,7}$, $L_{21,8}$, $L_{8,1}$, and $L_{13,2}$.
We nevertheless suspect that such counterexamples exist.
Moreover, since we currently have no strong reason to believe otherwise,
we tentatively propose the following conjecture:

\begin{conjecture}\label{conj:infiniteNoCore}
Every lens space admits infinitely many one-vertex triangulations with no core edges.
\end{conjecture}

For our other two problems of interest, our targeted search was again not always successful.
As mentioned in Section~\ref{subsec:remTunnel}, we were not able to find
a one-vertex ideal triangulation of the $5_2$ knot with no tunnel edges.
We also failed to find triangulations with no Seifert fibre edges for the following small Seifert fibre spaces:
$\mathcal{S}\left( \frac{1}{2}, \frac{1}{3}, -\frac{2}{3} \right)$,
$\mathcal{S}\left( \frac{1}{2}, \frac{1}{3}, -\frac{1}{3} \right)$,
$\mathcal{S}\left( \frac{1}{2}, \frac{1}{3}, -\frac{3}{4} \right)$,
$\mathcal{S}\left( \frac{1}{2}, \frac{1}{3}, -\frac{4}{5} \right)$, and
$\mathcal{S}\left( \frac{1}{2}, \frac{1}{3}, \frac{1}{3} \right)$.
Despite this, we propose the following conjectures:

\begin{conjecture}\label{conj:infiniteNoTunnel}
Every knot with tunnel number one admits infinitely many one-vertex ideal triangulations with no tunnel edges.
\end{conjecture}

\begin{conjecture}\label{conj:infiniteNoFibre}
For every irreducible small Seifert fibre space that is neither a lens space nor a prism manifold,
there exist infinitely many one-vertex triangulations with no Seifert fibre edges.
\end{conjecture}

\subsection{Future applications}\label{subsec:future}

Elementary moves such as 2-3 and 3-2 moves have many computational applications beyond how we used them in this paper.
The two most common are:
\begin{enumerate}[label={(\arabic*)}]
\item using elementary moves to improve a triangulation,
which is often critical for making exponential-time computations feasible; and
\item finding a sequence of moves that transforms a triangulation $\mathcal{T}$ into another triangulation $\mathcal{T}'$,
which gives a computational proof that $\mathcal{T}$ and $\mathcal{T}'$ are homeomorphic.
\end{enumerate}
Finding the right sequence of moves can be difficult,
so a targeted search could also be useful in these other settings.

Indeed, one of the main difficulties in these other settings is that increasing the number of tetrahedra is sometimes unavoidable~\cite{Burton2011arXiv}:
\begin{enumerate}[label={(\arabic*)}]
\item There are many triangulations $\mathcal{T}$ of the $3$-sphere such that to
simplify $\mathcal{T}$ to a minimal triangulation via 2-3 and 3-2 moves, we must visit
at least one intermediate triangulation with two more tetrahedra than $\mathcal{T}$.
There is also a triangulation of a graph manifold for which such simplification requires \emph{three} additional tetrahedra.
\item In the census of minimal triangulations with up to $9$ tetrahedra,
there are many $3$-manifolds for which two additional tetrahedra are required to
find sequences of 2-3 and 3-2 moves that connect all pairs of minimal triangulations.
There is also one $3$-manifold---the lens space $L_{3,1}$---for which \emph{three}
additional tetrahedra are required to connect all the minimal triangulations.
\end{enumerate}

Beyond the cases where we are forced to increase the number of tetrahedra,
there are also situations where we might actually \emph{want} to increase this number,
provided this improves the triangulation with respect to another quantity.
One such quantity is the \emph{treewidth}, which is important because
a number of algorithms in $3$-manifold topology are known to be fixed-parameter tractable
in the treewidth~\cite{BurtonDowney2017,BLPS2016FPTMorse,BMS2018,BurtonPettersson2014FPT,BurtonSpreer2013}.
For such algorithms, the benefit of reducing the treewidth significantly outweighs
the cost that would arise from increasing the number of tetrahedra.
This has practical relevance because, as observed by Husz\'{a}r and Spreer~\cite[Corollary~20]{HuszarSpreer2019},
minimal triangulations (which minimise the number of tetrahedra) do not necessarily minimise the treewidth;
here, we mention two examples of this, the first of which was given by Husz\'{a}r and Spreer in~\cite{HuszarSpreer2019}:
\begin{itemize}
\item Let $\mathcal{M}$ denote the Poincar\'{e} homology sphere;
as mentioned in Section~\ref{subsec:remFibre}, $\mathcal{M}$ admits the Seifert fibration
$\mathcal{S}\left( \frac{1}{2}, \frac{1}{3}, -\frac{4}{5} \right)$.
The unique minimal triangulation for $\mathcal{M}$ is given by the isomorphism signature \texttt{fvPQcdecedekrsnrs};
this is a $5$-tetrahedron triangulation with treewidth $4$.
However, the isomorphism signature \texttt{hLALAkcbbeffggqqnnmxkk} gives
a $7$-tetrahedron triangulation of $\mathcal{M}$ with treewidth $2$.
\item In Section~\ref{sec:intro}, we mentioned that there is a Seifert fibre space---to be precise,
the fibration is given by $\mathcal{S}\left( \frac{1}{2}, \frac{1}{3}, -\frac{9}{11} \right)$---whose
unique minimal triangulation is given by the isomorphism signature \texttt{iLLLPQcbcgffghhhtsmhgosof};
this is an $8$-tetrahedron triangulation with treewidth $4$.
However, we can also triangulate this $3$-manifold using the standard prism-and-layering construction;
this gives the isomorphism signature
\texttt{jLAMLLQbcbdeghhiixxnxxjqisj}, which represents a $9$-tetrahedron triangulation with treewidth $2$.
\end{itemize}

The main takeaway from this discussion is that sometimes we might
either want or need to increase the number of tetrahedra.
As mentioned in Section~\ref{sec:intro}, increasing the number of tetrahedra is
accompanied by a super-exponential increase in the number of triangulations;
even if we fix a $3$-manifold, this growth is still at least exponential
(whether this is actually super-exponential remains open).
This means that exhaustive search techniques quickly run into problems with not only running time,
but also (often more importantly) memory management.
However, our work in this paper suggests that, with the right heuristics, we may be able to circumvent these problems;
the (still unresolved) challenge is to actually devise such heuristics in the settings mentioned above.

\bibliography{TriangCounterexRefs}

\appendix

\section{More triangulations with no core edges}\label{appen:remCoreMore}

In Section~\ref{subsec:remCore}, we mentioned that in addition to finding a triangulation with no core edges for the $3$-sphere,
we also found such triangulations for the following lens spaces: $L_{6,1}$, $L_{9,2}$, $L_{11,2}$, $L_{13,3}$, and $L_{16,3}$.
For these five additional cases, it was always enough to just run Algorithm~\ref{algm:targEnum} with $x=0$ and $12$ concurrent processes;
the results are as follows:
\begin{itemize}
\item For $L_{6,1}$, we began with the isomorphism signature \texttt{dLQbcbchhjs}.
Figure~\ref{fig:remCorL61} summarises the result of running Algorithm~\ref{algm:targEnum}.
We then used a breadth-first search to simplify down to a $12$-tetrahedron example with isomorphism signature
\[
\texttt{mLvwLLQQQadjljgilkilkknaqlabamaqacv}.
\]
	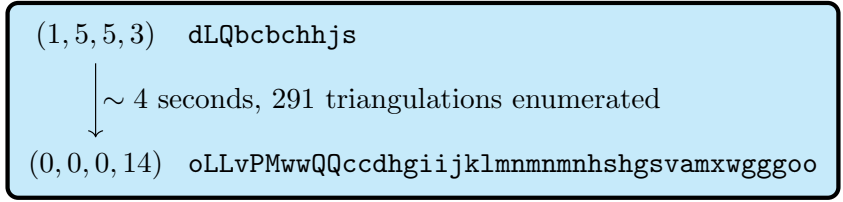
\begin{figure}[htbp]
	\centering
	\begin{tikzcd}[
		column sep=0ex,
		row sep=large,
		/tikz/column 2/.append style={anchor=base west},
		every label/.append style={font=\normalsize},
		resultsbox,]
	(1,5,5,3)
	\ar{d}{
	\text{$\sim4$ seconds, $291$ triangulations enumerated}}
	&
	\texttt{dLQbcbchhjs}
	\\
	(0,0,0,14)
	&
	\texttt{oLLvPMwwQQccdhgiijklmnmnmnhshgsvamxwgggoo}
	\end{tikzcd}
	\caption{Removing core edges from \texttt{dLQbcbchhjs} ($L_{6,1}$); $x=0$, $12$ processes.}
	\label{fig:remCorL61}
	\end{figure}
\item For $L_{9,2}$, we began with the isomorphism signature \texttt{dLQbcbchhjw}.
Figure~\ref{fig:remCorL92} summarises the result of running Algorithm~\ref{algm:targEnum} twice.
We then used a breadth-first search to simplify down to a $13$-tetrahedron example with isomorphism signature
\[
\texttt{nvLLMMAwQkcdghghjikmllmmnkpmigoriulkkn}.
\]
	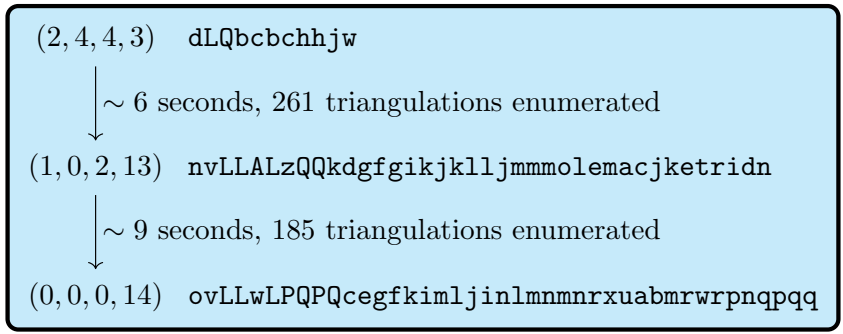
\begin{figure}[htbp]
	\centering
	\begin{tikzcd}[
		column sep=0ex,
		row sep=large,
		/tikz/column 2/.append style={anchor=base west},
		every label/.append style={font=\normalsize},
		resultsbox,]
	(2,4,4,3)
	\ar{d}{
	\text{$\sim6$ seconds, $261$ triangulations enumerated}}
	&
	\texttt{dLQbcbchhjw}
	\\
	(1,0,2,13)
	\ar{d}{
	\text{$\sim9$ seconds, $185$ triangulations enumerated}}
	&
	\texttt{nvLLALzQQkdgfgikjklljmmmolemacjketridn}
	\\
	(0,0,0,14)
	&
	\texttt{ovLLwLPQPQcegfkimljinlmnmnrxuabmrwrpnqpqq}
	\end{tikzcd}
	\caption{Removing core edges from \texttt{dLQbcbchhjw} ($L_{9,2}$); $x=0$, $12$ processes.}
	\label{fig:remCorL92}
	\end{figure}
\item For $L_{11,2}$, we began with the isomorphism signature \texttt{eLAkbcbddhhjhr}.
Figure~\ref{fig:remCorL11_2} summarises the result of running Algorithm~\ref{algm:targEnum} twice.
We then used a breadth-first search to simplify down to a $16$-tetrahedron example with isomorphism signature
\[
\texttt{qLvwLPLvQQQkadhgkiknopknopponjajcxacsvvocvvvvc}.
\]
	\begin{figure}[htbp]
	\centering
	\begin{tikzcd}[
		column sep=0ex,
		row sep=large,
		/tikz/column 2/.append style={anchor=base west},
		every label/.append style={font=\normalsize},
		resultsbox,]
	(2,5,6,2)
	\ar{d}{
	\text{$\sim19$ seconds, $790$ triangulations enumerated}}
	&
	\texttt{eLAkbcbddhhjhr}
	\\
	(1,1,3,16)
	\ar{d}{
	\text{$\sim11$ seconds, $192$ triangulations enumerated}}
	&
	\texttt{qvLLvLALQQQkdgihimkmkolnonpppoliggfsfjpodilmfg}
	\\
	(0,0,0,19)
	&
	\texttt{tLvAvMvwMwQQQkadfhikmlprsopqpsqrqsnaqaamvktmqpauamigwv}
	\end{tikzcd}
	\caption{Removing core edges from \texttt{eLAkbcbddhhjhr} ($L_{11,2}$); $x=0$, $12$ processes.}
	\label{fig:remCorL11_2}
	\end{figure}
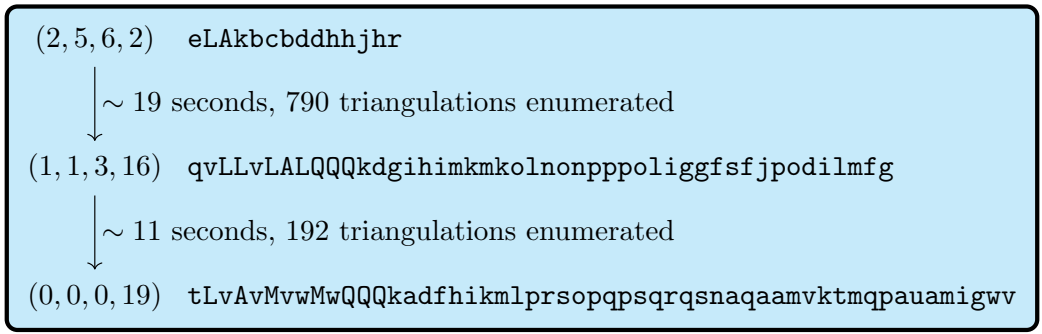
\item For $L_{13,3}$, we began with the isomorphism signature \texttt{eLAkbcbddhhjqn}.
Figure~\ref{fig:remCorL13_3} summarises the result of running Algorithm~\ref{algm:targEnum} twice.
We then used a breadth-first search to simplify down to a $17$-tetrahedron example with isomorphism signature
\[
\texttt{rvLLALwwLQQQccdgfkhkoqpqpononpqnkpisnsaaaanngngns}.
\]
	\begin{figure}[htbp]
	\centering
	\begin{tikzcd}[
		column sep=0ex,
		row sep=large,
		/tikz/column 2/.append style={anchor=base west},
		every label/.append style={font=\normalsize},
		resultsbox,]
	(2,4,5,4)
	\ar{d}{
	\text{$\sim44$ seconds, $1283$ triangulations enumerated}}
	&
	\texttt{eLAkbcbddhhjqn}
	\\
	(1,0,4,18)
	\ar{d}{
	\text{$\sim90$ seconds, $959$ triangulations enumerated}}
	&
	\texttt{svLALvMvLQQQQcefflhormrmpqnnpqrqnnaiaktaualakkvgoro}
	\\
	(0,0,0,19)
	&
	\texttt{tvLALvwLMMAQQkceffljopqqrlprosopssnnaiafaaddhtkrkakkut}
	\end{tikzcd}
	\caption{Removing core edges from \texttt{eLAkbcbddhhjqn} ($L_{13,3}$); $x=0$, $12$ processes.}
	\label{fig:remCorL13_3}
	\end{figure}
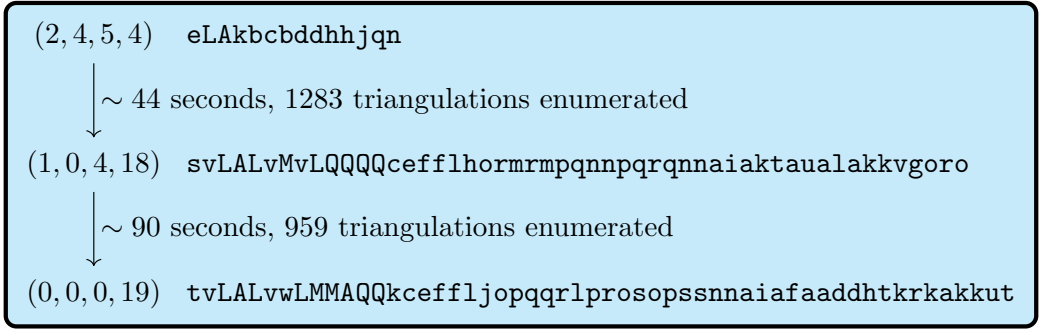
\item For $L_{16,3}$, we began with the isomorphism signature \texttt{fLAMcbcbdeehhjhxj}.
Figure~\ref{fig:remCorL16_3} summarises the result of running Algorithm~\ref{algm:targEnum} twice.
We then used a breadth-first search to simplify down to a $20$-tetrahedron example with isomorphism signature
\[
\texttt{uLvwvAvLzQPQQQcadigooqmoslrprtsprtstnaaoqovfctkfkfqlkfxta}.
\]
	\begin{figure}[htbp]
	\centering
	\begin{tikzcd}[
		column sep=0ex,
		row sep=large,
		/tikz/column 2/.append style={anchor=base west},
		every label/.append style={font=\normalsize},
		resultsbox,]
	(2,5,7,5)
	\ar{d}{
	\text{$\sim44$ seconds, $3506$ triangulations enumerated}}
	&
	\texttt{fLAMcbcbdeehhjhxj}
	\\
	(1,0,2,20)
	\ar{d}{
	\text{$\sim12$ seconds, $396$ triangulations enumerated}}
	&
	\texttt{uLvvwAwwzMQQAQccghkilpnqosnrrsqotsttpagcvaaahapggvcfbmgxm}
	\\
	(0,0,0,21)
	&
	\texttt{vLvvwLAzzPAMQQQcghimnqorlonosprutstuutabnaaadamdgffgfdjfjpm}
	\end{tikzcd}
	\caption{Removing core edges from \texttt{fLAMcbcbdeehhjhxj} ($L_{16,3}$); $x=0$, $12$ processes.}
	\label{fig:remCorL16_3}
	\end{figure}
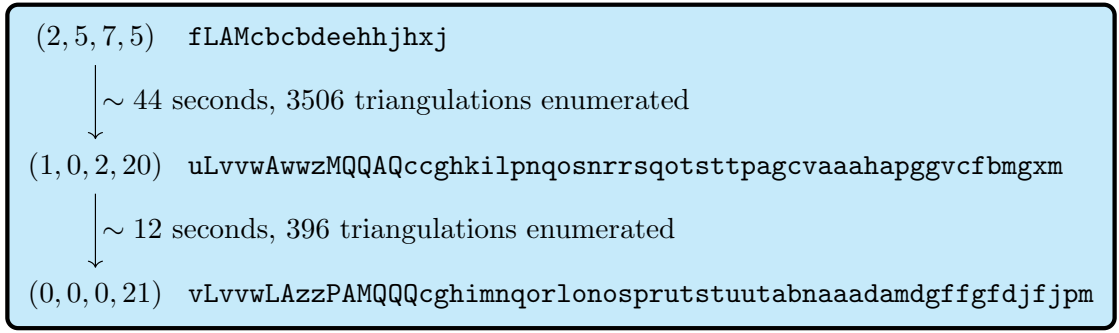
\end{itemize}

\section{More triangulations with no tunnel edges}\label{appen:remTunMore}

In Section~\ref{subsec:remTunnel}, we mentioned that we found triangulations with no tunnel edges for
all but one of the prime knots with crossing number up to $7$,
but did not give detailed results for the cases with crossing number $6$ and $7$.
We present these additional results here:

\begin{itemize}
\item For the $6_1$ knot, we began with the isomorphism signature \texttt{gLLMQccefeffdfeqldg}.
Figure~\ref{fig:remTun6_1} summarises the results of running Algorithm~\ref{algm:targEnum} four times.
	\begin{figure}[htbp]
	\centering
	\begin{tikzcd}[
		column sep=0ex,
		row sep=large,
		/tikz/column 2/.append style={anchor=base west},
		every label/.append style={font=\normalsize},
		resultsbox,]
	(4,2,4,6)
	\ar{d}{
	\text{$\sim43$ seconds, $105$ triangulations enumerated}}
	&
	\texttt{gLLMQccefeffdfeqldg}
	\\
	(3,3,5,7)
	\ar{d}{
	\text{$\sim144$ seconds, $377$ triangulations enumerated}}
	&
	\texttt{hLLAPkcdedfggghslfcffg}
	\\
	(2,2,4,9)
	\ar{d}{
	\text{$\sim260$ seconds, $475$ triangulations enumerated}}
	&
	\texttt{jLLLLQQbcghiighihxsmoorgogf}
	\\
	(1,0,2,13)
	\ar{d}{
	\text{$\sim182$ seconds, $282$ triangulations enumerated}}
	&
	\texttt{nvLALMvQQkcdfihjhkkllmmmnwaipckllhmljc}
	\\
	(0,0,0,13)
	&
	\texttt{nLvAzLAMQkbfeghijkmkmllmmppavkhhaxjorb}
	\end{tikzcd}
	\caption{Removing tunnel edges from \texttt{gLLMQccefeffdfeqldg} ($6_1$ knot); $x=0$, $12$ processes.}
	\label{fig:remTun6_1}
	\end{figure}
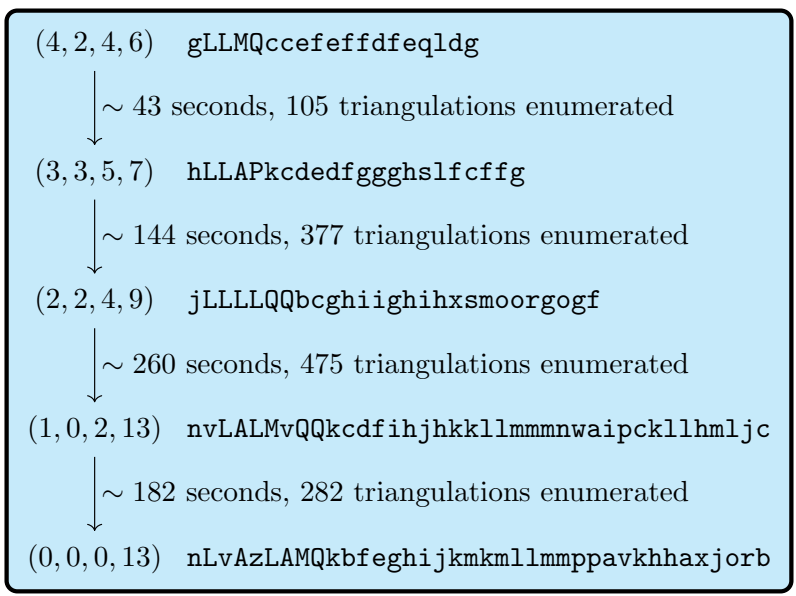
\item For the $6_2$ knot, we began with the isomorphism signature \texttt{fLLQcbeddeeedgfid}.
Figure~\ref{fig:remTun6_2} summarises the results of running Algorithm~\ref{algm:targEnum} four times.
	\begin{figure}[htbp]
	\centering
	\begin{tikzcd}[
		column sep=0ex,
		row sep=large,
		/tikz/column 2/.append style={anchor=base west},
		every label/.append style={font=\normalsize},
		resultsbox,]
	(4,2,4,5)
	\ar{d}{
	\text{$\sim14$ seconds, $74$ triangulations enumerated}}
	&
	\texttt{fLLQcbeddeeedgfid}
	\\
	(3,3,5,6)
	\ar{d}{
	\text{$\sim34$ seconds, $139$ triangulations enumerated}}
	&
	\texttt{gLLMQbcdefffdoweuoc}
	\\
	(2,3,6,6)
	\ar{d}{
	\text{$\sim127$ seconds, $406$ triangulations enumerated}}
	&
	\texttt{gLLMQbcdefffpnbhunw}
	\\
	(1,0,2,10)
	\ar{d}{
	\text{$\sim90$ seconds, $200$ triangulations enumerated}}
	&
	\texttt{kvLLAQPkegfhfihijjjfmpxeupoxld}
	\\
	(0,0,0,11)
	&
	\texttt{lLLvAMPQccefhigijjkkkuiaaofvaqems}
	\end{tikzcd}
	\caption{Removing tunnel edges from \texttt{fLLQcbeddeeedgfid} ($6_2$ knot); $x=0$, $12$ processes.}
	\label{fig:remTun6_2}
	\end{figure}
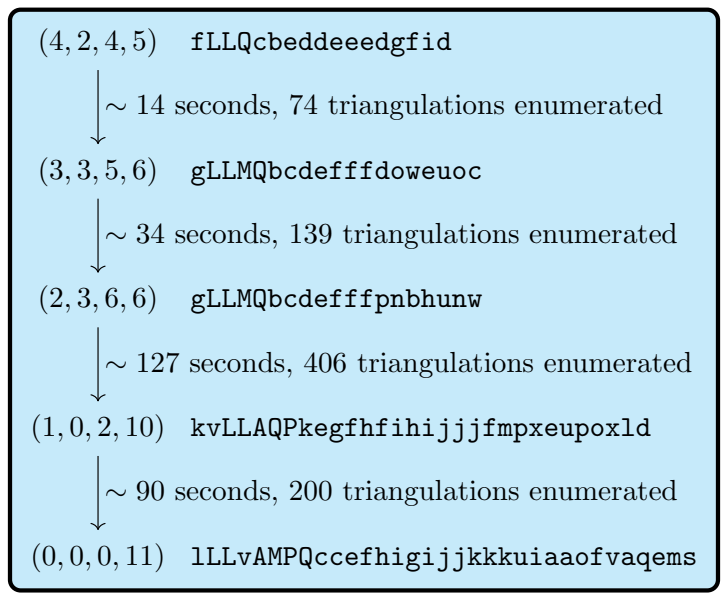
\item For the $6_3$ knot, we began with the isomorphism signature \texttt{eLPkbcddddcwjb}.
Figure~\ref{fig:remTun6_3} summarises the results of running Algorithm~\ref{algm:targEnum} three times.
	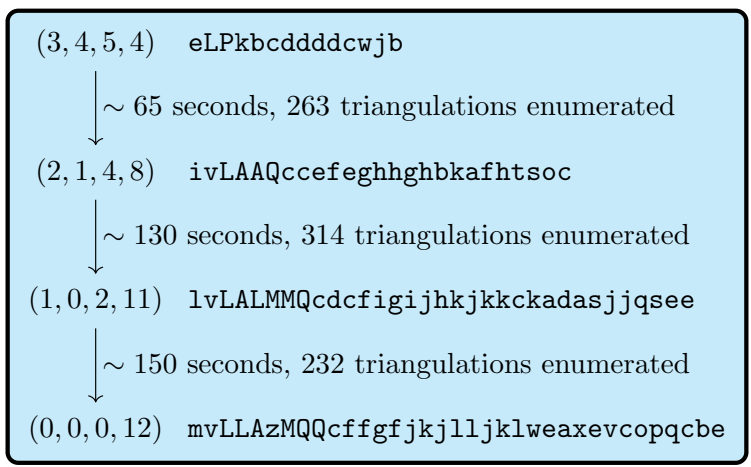
\begin{figure}[htbp]
	\centering
	\begin{tikzcd}[
		column sep=0ex,
		row sep=large,
		/tikz/column 2/.append style={anchor=base west},
		every label/.append style={font=\normalsize},
		resultsbox,]
	(3,4,5,4)
	\ar{d}{
	\text{$\sim65$ seconds, $263$ triangulations enumerated}}
	&
	\texttt{eLPkbcddddcwjb}
	\\
	(2,1,4,8)
	\ar{d}{
	\text{$\sim130$ seconds, $314$ triangulations enumerated}}
	&
	\texttt{ivLAAQccefeghhghbkafhtsoc}
	\\
	(1,0,2,11)
	\ar{d}{
	\text{$\sim150$ seconds, $232$ triangulations enumerated}}
	&
	\texttt{lvLALMMQcdcfigijhkjkkckadasjjqsee}
	\\
	(0,0,0,12)
	&
	\texttt{mvLLAzMQQcffgfjkjlljklweaxevcopqcbe}
	\end{tikzcd}
	\caption{Removing tunnel edges from \texttt{eLPkbcddddcwjb} ($6_3$ knot); $x=0$, $12$ processes.}
	\label{fig:remTun6_3}
	\end{figure}
\item For the $(7,2)$ torus knot (i.e., the $7_1$ knot), we began with the isomorphism signature
\[
\texttt{eLAkbccddaekln}.
\]
Figure~\ref{fig:remTunTorus72} summarises the results of running Algorithm~\ref{algm:targEnum} twice.
	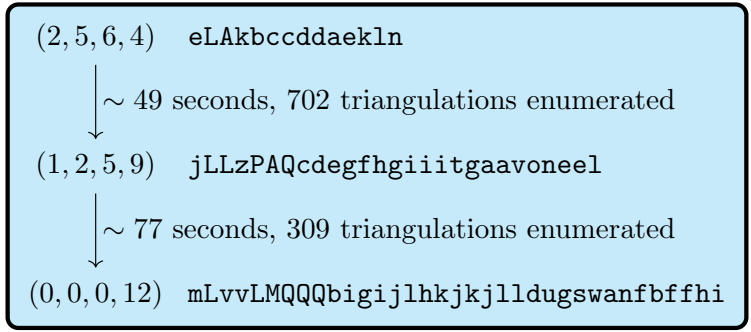
\begin{figure}[htbp]
	\centering
	\begin{tikzcd}[
		column sep=0ex,
		row sep=large,
		/tikz/column 2/.append style={anchor=base west},
		every label/.append style={font=\normalsize},
		resultsbox,]
	(2,5,6,4)
	\ar{d}{
	\text{$\sim49$ seconds, $702$ triangulations enumerated}}
	&
	\texttt{eLAkbccddaekln}
	\\
	(1,2,5,9)
	\ar{d}{
	\text{$\sim77$ seconds, $309$ triangulations enumerated}}
	&
	\texttt{jLLzPAQcdegfhgiiitgaavoneel}
	\\
	(0,0,0,12)
	&
	\texttt{mLvvLMQQQbigijlhkjkjlldugswanfbffhi}
	\end{tikzcd}
	\caption{Removing tunnel edges from \texttt{eLAkbccddaekln} ($(7,2)$ torus knot); $x=0$, $12$ processes.}
	\label{fig:remTunTorus72}
	\end{figure}
\item For the $7_2$ knot, we began with the isomorphism signature \texttt{eLAkbccddmejln}.
Figure~\ref{fig:remTun7_2} summarises the results of running Algorithm~\ref{algm:targEnum} three times.
	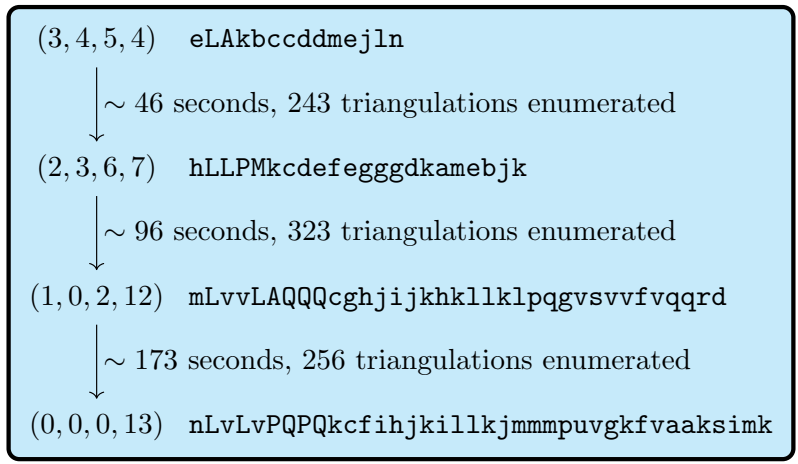
\begin{figure}[htbp]
	\centering
	\begin{tikzcd}[
		column sep=0ex,
		row sep=large,
		/tikz/column 2/.append style={anchor=base west},
		every label/.append style={font=\normalsize},
		resultsbox,]
	(3,4,5,4)
	\ar{d}{
	\text{$\sim46$ seconds, $243$ triangulations enumerated}}
	&
	\texttt{eLAkbccddmejln}
	\\
	(2,3,6,7)
	\ar{d}{
	\text{$\sim96$ seconds, $323$ triangulations enumerated}}
	&
	\texttt{hLLPMkcdefegggdkamebjk}
	\\
	(1,0,2,12)
	\ar{d}{
	\text{$\sim173$ seconds, $256$ triangulations enumerated}}
	&
	\texttt{mLvvLAQQQcghjijkhkllklpqgvsvvfvqqrd}
	\\
	(0,0,0,13)
	&
	\texttt{nLvLvPQPQkcfihjkillkjmmmpuvgkfvaaksimk}
	\end{tikzcd}
	\caption{Removing tunnel edges from \texttt{eLAkbccddmejln} ($7_2$ knot); $x=0$, $12$ processes.}
	\label{fig:remTun7_2}
	\end{figure}
\item For the $7_3$ knot, we began with the isomorphism signature \texttt{fLLQcbeddeeddgfua}.
Figure~\ref{fig:remTun7_3} summarises the results of running Algorithm~\ref{algm:targEnum} four times.
	\begin{figure}[htbp]
	\centering
	\begin{tikzcd}[
		column sep=0ex,
		row sep=large,
		/tikz/column 2/.append style={anchor=base west},
		every label/.append style={font=\normalsize},
		resultsbox,]
	(4,2,4,5)
	\ar{d}{
	\text{$\sim15$ seconds, $85$ triangulations enumerated}}
	&
	\texttt{fLLQcbeddeeddgfua}
	\\
	(3,3,6,7)
	\ar{d}{
	\text{$\sim32$ seconds, $84$ triangulations enumerated}}
	&
	\texttt{hLvMQkbefefgggpptxtmwr}
	\\
	(2,4,7,9)
	\ar{d}{
	\text{$\sim238$ seconds, $452$ triangulations enumerated}}
	&
	\texttt{jvLLMQQcdfgfihhiikjiatavvcj}
	\\
	(1,0,1,12)
	\ar{d}{
	\text{$\sim179$ seconds, $194$ triangulations enumerated}}
	&
	\texttt{mLLvAMLQQbeghgikkljlklxptavratieivs}
	\\
	(0,0,0,12)
	&
	\texttt{mLLLvMQMQceghhikijlklliixvvjiroaplj}
	\end{tikzcd}
	\caption{Removing tunnel edges from \texttt{fLLQcbeddeeddgfua} ($7_3$ knot); $x=0$, $12$ processes.}
	\label{fig:remTun7_3}
	\end{figure}
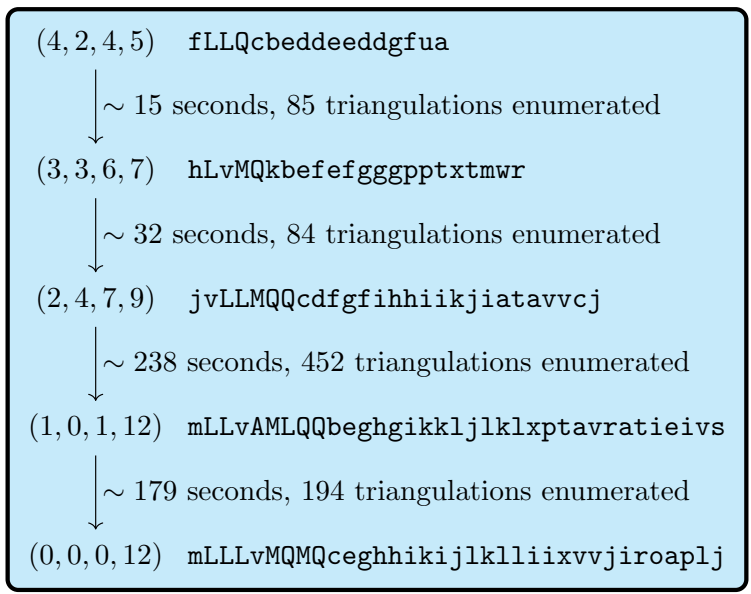
\item For the $7_4$ knot, we began with the isomorphism signature \texttt{gLLAQbdedfffdwmfarv},
which has complexity $(4,2,4,6)$.
To reduce the number of tunnel edges to $3$, we needed to run Algorithm~\ref{algm:targEnum} with
$24$ concurrent processes ($12$ processes was not enough).
After that, it was enough to run Algorithm~\ref{algm:targEnum} three times with $12$ concurrent processes.
The results are summarised in Figure~\ref{fig:remTun7_4}.
	\begin{figure}[htbp]
	\centering
	\begin{tikzcd}[
		column sep=0ex,
		row sep=huge,
		/tikz/column 2/.append style={anchor=base west},
		every label/.append style={font=\normalsize},
		resultsbox,]
	(4,2,4,6)
	\ar{d}{\begin{array}{l}
	\text{$x=0$, $24$ concurrent processes}\\
	\text{$\sim124$ seconds, $399$ triangulations enumerated}
	\end{array}}
	&
	\texttt{fLLQcbeddeeddgfua}
	\\
	(3,2,5,7)
	\ar{d}{\begin{array}{l}
	\text{$x=0$, $12$ concurrent processes}\\
	\text{$\sim45$ seconds, $140$ triangulations enumerated}
	\end{array}}
	&
	\texttt{hLvAQkbefgefggpphemxwn}
	\\
	(2,2,6,9)
	\ar{d}{\begin{array}{l}
	\text{$x=0$, $12$ concurrent processes}\\
	\text{$\sim118$ seconds, $260$ triangulations enumerated}
	\end{array}}
	&
	\texttt{jvLLMQQdfghgfhiiicaheatebbg}
	\\
	(1,0,5,11)
	\ar{d}{\begin{array}{l}
	\text{$x=0$, $12$ concurrent processes}\\
	\text{$\sim462$ seconds, $593$ triangulations enumerated}
	\end{array}}
	&
	\texttt{lvLPLPMQcefggihjkjikkfauuqafljnue}
	\\
	(0,0,0,14)
	&
	\texttt{oLvwLLQPAQcadhgklijmmlknnnnankqaaaisrktqm}
	\end{tikzcd}
	\caption{Removing tunnel edges from \texttt{gLLAQbdedfffdwmfarv} ($7_4$ knot).}
	\label{fig:remTun7_4}
	\end{figure}
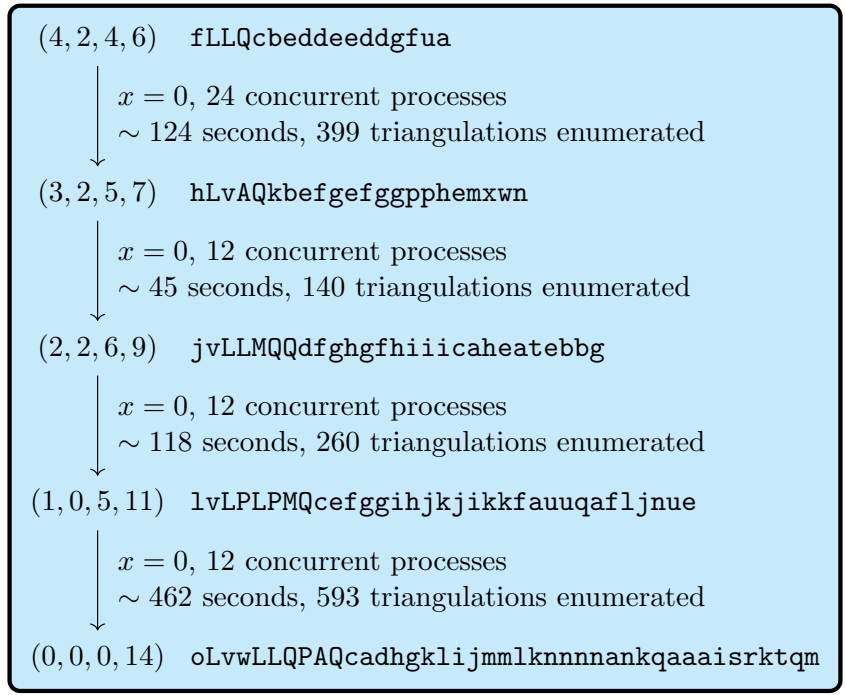
\item For the $7_5$ knot, we began with the isomorphism signature \texttt{hLLAMkbedefgggtlftavkb}.
Figure~\ref{fig:remTun7_5} summarises the results of running Algorithm~\ref{algm:targEnum} five times.
	\begin{figure}[htbp]
	\centering
	\begin{tikzcd}[
		column sep=0ex,
		row sep=large,
		/tikz/column 2/.append style={anchor=base west},
		every label/.append style={font=\normalsize},
		resultsbox,]
	(5,4,5,7)
	\ar{d}{
	\text{$\sim3$ seconds, $3$ triangulations enumerated}}
	&
	\texttt{hLLAMkbedefgggtlftavkb}
	\\
	(4,2,4,7)
	\ar{d}{
	\text{$\sim35$ seconds, $82$ triangulations enumerated}}
	&
	\texttt{hLLLQkcdefegggioaainfr}
	\\
	(3,2,4,8)
	\ar{d}{
	\text{$\sim141$ seconds, $276$ triangulations enumerated}}
	&
	\texttt{iLLLAQcbcfgghfhhhohaimjma}
	\\
	(2,2,4,9)
	\ar{d}{
	\text{$\sim195$ seconds, $330$ triangulations enumerated}}
	&
	\texttt{jLLLLQQbcgiihhighmsmksknnrf}
	\\
	(1,0,1,12)
	\ar{d}{
	\text{$\sim167$ seconds, $228$ triangulations enumerated}}
	&
	\texttt{mLvvAMMQQbfijhjhkjlkllmahcvkopnqtqb}
	\\
	(0,0,0,12)
	&
	\texttt{mLvLLMAQQbefgikjlijlklmpmkvxfavvhur}
	\end{tikzcd}
	\caption{Removing tunnel edges from \texttt{hLLAMkbedefgggtlftavkb} ($7_5$ knot); $x=0$, $12$ processes.}
	\label{fig:remTun7_5}
	\end{figure}
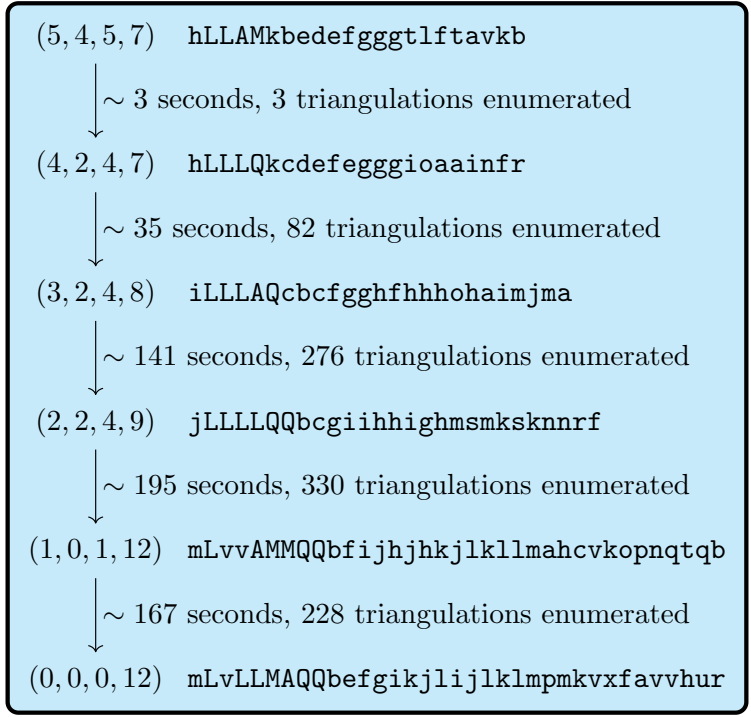
\item For the $7_6$ knot, we began with the isomorphism signature \texttt{iLLvQQccdeggfhhhdoldgrgnk}.
Figure~\ref{fig:remTun7_6} summarises the results of running Algorithm~\ref{algm:targEnum} four times.
	\begin{figure}[htbp]
	\centering
	\begin{tikzcd}[
		column sep=0ex,
		row sep=large,
		/tikz/column 2/.append style={anchor=base west},
		every label/.append style={font=\normalsize},
		resultsbox,]
	(4,2,5,8)
	\ar{d}{
	\text{$\sim212$ seconds, $384$ triangulations enumerated}}
	&
	\texttt{iLLvQQccdeggfhhhdoldgrgnk}
	\\
	(3,1,5,12)
	\ar{d}{
	\text{$\sim300$ seconds, $191$ triangulations enumerated}}
	&
	\texttt{mLLLPwMPQacffghjkljkllnsgsjaajawbpp}
	\\
	(2,1,3,12)
	\ar{d}{
	\text{$\sim440$ seconds, $325$ triangulations enumerated}}
	&
	\texttt{mLLPPvPMQaceffgikjjlllnsxgsacaxaccv}
	\\
	(1,0,1,13)
	\ar{d}{
	\text{$\sim384$ seconds, $215$ triangulations enumerated}}
	&
	\texttt{nLLPPLvMQkaceffghlkmlmmlnsxgsareebffbb}
	\\
	(0,0,0,13)
	&
	\texttt{nLLPPLvQPkaceffghkjlllmmnsxgsarxdafnmj}
	\end{tikzcd}
	\caption{Removing tunnel edges from \texttt{iLLvQQccdeggfhhhdoldgrgnk} ($7_6$ knot); $x=0$, $12$ processes.}
	\label{fig:remTun7_6}
	\end{figure}
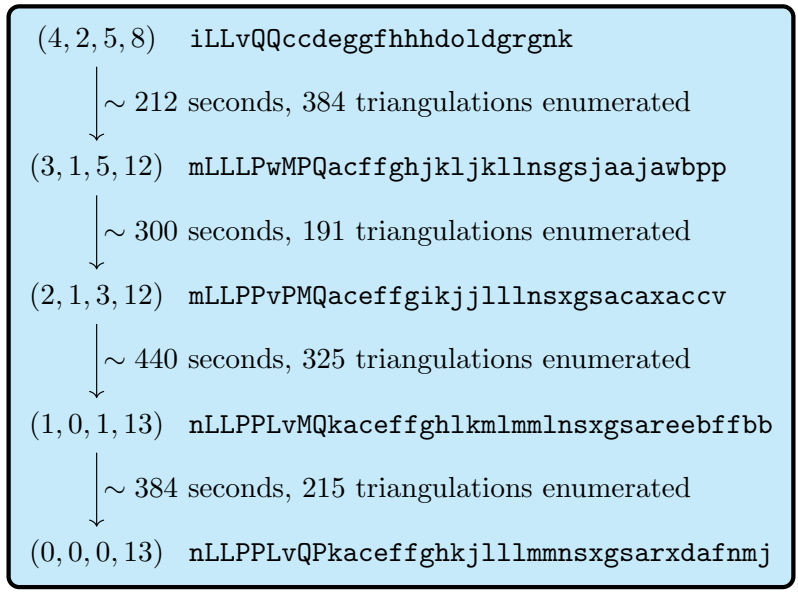
\item For the $7_7$ knot, we began with the isomorphism signature \texttt{iLLPzQcbcffghghhmomtdirhx}, which has complexity $(4,2,4,8)$.
To reduce the number of tunnel edges to $3$, we needed to run Algorithm~\ref{algm:targEnum} with $24$ concurrent processes ($12$ processes was not enough).
After that, it was enough to run Algorithm~\ref{algm:targEnum} three times with $12$ concurrent processes.
The results are summarised in Figure~\ref{fig:remTun7_7}.
	\begin{figure}[htbp]
	\centering
	\begin{tikzcd}[
		column sep=0ex,
		row sep=huge,
		/tikz/column 2/.append style={anchor=base west},
		every label/.append style={font=\normalsize},
		resultsbox,]
	(4,2,4,8)
	\ar{d}{\begin{array}{l}
	\text{$x=0$, $24$ concurrent processes}\\
	\text{$\sim239$ seconds, $313$ triangulations enumerated}
	\end{array}}
	&
	\texttt{iLLPzQcbcffghghhmomtdirhx}
	\\
	(3,2,6,10)
	\ar{d}{\begin{array}{l}
	\text{$x=0$, $12$ concurrent processes}\\
	\text{$\sim595$ seconds, $455$ triangulations enumerated}
	\end{array}}
	&
	\texttt{kLLvMMQkcdgighjijijioeckracdst}
	\\
	(2,0,4,13)
	\ar{d}{\begin{array}{l}
	\text{$x=0$, $12$ concurrent processes}\\
	\text{$\sim107$ seconds, $97$ triangulations enumerated}
	\end{array}}
	&
	\texttt{nvLLvAQQPkeghilikjkljmmmoqmakockfkchah}
	\\
	(1,2,6,12)
	\ar{d}{\begin{array}{l}
	\text{$x=0$, $12$ concurrent processes}\\
	\text{$\sim1078$ seconds, $694$ triangulations enumerated}
	\end{array}}
	&
	\texttt{mLLvzMAQQccgihkijjllkliviwvavxxaabf}
	\\
	(0,0,0,15)
	&
	\texttt{pLLLLvPMQPQccfgmhjknlmlomooivarlfaiaifnafix}
	\end{tikzcd}
	\caption{Removing tunnel edges from \texttt{iLLPzQcbcffghghhmomtdirhx} ($7_7$ knot).}
	\label{fig:remTun7_7}
	\end{figure}
\end{itemize}

\end{document}